\documentclass{gOMS2e}

\usepackage{epstopdf}
\usepackage{subfigure}
\usepackage{graphicx}
\usepackage{psfrag}
\usepackage{subfigure}
\usepackage{url}
\usepackage{color}
\usepackage{epsfig}
\usepackage{array,amsmath,amssymb}
\usepackage{float}

\newcommand{\rset}{\mathbb{R}}

\providecommand{\norm}[1]{\lVert#1\rVert}

\theoremstyle{plain}
\newtheorem{theorem}{Theorem}[section]
\newtheorem{corollary}[theorem]{Corollary}
\newtheorem{lemma}[theorem]{Lemma}

\newtheorem{assumption}[theorem]{Assumption}
\theoremstyle{definition}
\newtheorem{definition}{Definition}
\theoremstyle{remark}
\newtheorem{remark}{Remark}
\theoremstyle{example}
\newtheorem{example}{Example}

\begin{document}

\articletype{RESEARCH ARTICLE}

\title{Complexity  of first order inexact Lagrangian
and penalty methods for conic convex programming}

\author{
\name{I. Necoara\textsuperscript{a}$^{\ast}$
\thanks{$^\ast$Corresponding author, email:
ion.necoara@acse.pub.ro. The research leading to these results has
received funding from UEFISCDI Romania, PNII-RU- TE, project
MoCOBiDS, no. 176/01.10.2015. It also presents research results of
the Belgian Network DYSCO funded by the Interuniversity Attraction
Poles Programme initiated by the Belgian State, and of the Concerted
Research Action programme supported by the Federation
Wallonia-Brussels, no. ARC 14/19-060. Support from two WBI-Romanian
Academy grants is also acknowledged. The authors thank Prof. Yu.
Nesterov for inspiring discussions.} and A.
Patrascu\textsuperscript{a} and F. Glineur\textsuperscript{b}}
\affil{\textsuperscript{a}Automatic Control and  Systems Engineering
Department, University Politehnica Bucharest, 060042 Bucharest,
Romania. \textsuperscript{b}Center for Operations Research and
Econometrics, Catholic University of Louvain, B-1348
Louvain-la-Neuve, Belgium
}
}

\maketitle

\begin{abstract}
In this paper we present a complete iteration complexity  analysis
of inexact first order   Lagrangian and penalty methods for solving
cone constrained convex problems that have or may not have optimal
Lagrange multipliers that close the duality gap. We first assume the
existence of optimal Lagrange multipliers and study primal-dual
first order methods based on inexact information and augmented
Lagrangian smoothing or Nesterov type smoothing. For inexact (fast)
gradient augmented Lagrangian methods we derive a total
computational complexity of $\mathcal{O}\left(
\frac{1}{\epsilon}\right)$ projections onto a simple primal  set in
order to attain an $\epsilon-$optimal solution of the conic convex
problem. For the inexact fast gradient method combined with Nesterov
type smoothing we derive computational complexity $\mathcal{O}\left(
\frac{1}{\epsilon^{3/2}}\right)$ projections onto the same set.
Then, we assume that optimal Lagrange multipliers for the cone
constrained convex problem might not exist, and analyze the fast
gradient method for solving penalty reformulations of the problem.
For the fast gradient method combined with penalty framework we also
derive a total  computational complexity of $\mathcal{O}\left(
\frac{1}{\epsilon^{3/2}}\right)$ projections onto a simple primal
set to attain an $\epsilon-$optimal solution for the original
problem.
\end{abstract}

\begin{keywords}
conic convex problems, smooth (augmented) dual functions, penalty functions, (augmented) dual first order methods, penalty fast gradient methods, approximate primal solution,  computational complexity.
\end{keywords}

\begin{classcode}
90C25; 90C46; 68Q25; 65K05.
\end{classcode}


\section{Introduction}
\noindent Many recent engineering and economical applications  can
be posed as  large-scale conic convex problems and thus the interest
for scalable algorithms with inexpensive iterations is continuously
increasing. For instance, in the recent optimization literature,
first order methods gained much attention since they present cheap
iterations and are usually adequate  {for
large-scale} convex setting. In the constrained case, when there are
conic complicated constraints, many first order algorithms are
combined with duality  or penalty strategies. For example,  in
\cite{LanLu:11} various smooth and nonsmooth formulations are
provided for  cone programming, and through application of first
order methods (e.g. fast gradient or mirror descent) on the
corresponding reformulations of the optimality conditions as
optimization problems, an $\epsilon$-optimal solution is obtained in
$\mathcal{O}\left(\frac{1}{\epsilon}\right)$ projections onto a
simple primal set. For conic constrained convex  problems, quadratic
penalty strategies are combined with fast gradient method in
\cite{LanMon:13}.  Under the assumptions  of smooth objective
function and  existence of a finite optimal Lagrange multiplier, the
first order  quadratic penalty method in \cite{LanMon:13} requires
$\mathcal{O}(\frac{1}{\epsilon^2})$ fast gradient iterations.
Moreover, using a regularization of the original problem with a
strongly convex term, this method requires
$\mathcal{O}(\frac{1}{\epsilon}\log \left(
\frac{1}{\epsilon}\right))$ fast gradient iterations. Recently,
other first order augmented Lagrangian methods are presented in
\cite{AybIye:13,LanMon:15,NedNec:14} and computational complexity
estimates of order $\mathcal{O}\left(\frac{1}{\epsilon}\right)$ are
obtained for smooth problems with bounded optimal Lagrange
multipliers. First order methods are also combined with duality and
Nesterov type smoothing in
\cite{BecTeb:12,BotHen:15,BotHen:15o,DevGli:12,NecSuy:08,QuoNec:15,QuoCev:14,YurQuo:15}
and convergence rates of order
$\mathcal{O}\left(\frac{1}{\epsilon}\right)$ in terms of dual
gradient evaluations are derived. Another interesting approach
relies on reformulation of conic constrained programming problems
into a monotone variational inequality and  {then}
designing various algorithms for solving these inequalities. This
approach can be found in \cite{Nes:07,Nem:04}, where different
primal-dual methods are devised for solving the variational
inequality under the boundedness assumption of the primal and dual
feasible sets. Recently,  the boundedness condition has been
eliminated in \cite{MonSva:10}.

\vspace{5pt}

\noindent \textbf{Motivation}. However, the following issues  can be
identified in the existing literature:

\noindent (a) Most of the existing papers on dual first order
methods  combined with smoothing techniques derive rate of
convergence results in terms of outer iterations (number of dual
gradient evaluations). However, we will show (see e.g., Theorem 3.4)
that one might choose an appropriate value of the smoothing
parameter such that after a single outer iteration an
$\epsilon-$solution can be obtained. Thus, convergence rates in
terms of outer iterations are not relevant in this case and it is
natural to analyze the overall complexity  of these methods that
also take into account the inner iterations (e.g., number of
projections onto the primal feasible set or  {number
of matrix-vector multiplications}).

\noindent (b) Moreover, from our knowledge,  there is no complete
analysis in the optimization literature regarding the overall
complexity of inexact dual first order methods based on augmented
Lagrangian smoothing and Nesterov smoothing and clarifying which
smoothing approach has a better behavior.

\noindent (c) Finally,   all the papers on Lagrangian and penalty
methods mentioned above make the strong assumption that there exists
an optimal Lagrange multiplier for the primal convex problem that
closes the duality gap. This property is usually guaranteed through
a Slater type condition, which in the large-scale setting is very
difficult to check computationally or even might not hold. Recently,
Nesterov developed in \cite{Nes:14} subgradient  methods for
nonsmooth convex problems with functional constraints without this
assumption on the existence of an optimal Lagrange multiplier  and
proved that  an $\epsilon$-optimal point can be attained after
$\mathcal{O}\left(\frac{1}{\epsilon^2}\right)$ subgradient
evaluations  for either the objective function or for a functional
constraint. Nesterov also asks in  \cite{Nes:14}  whether it is
possible to improve this convergence rate result under additional
smooth assumptions on the objective function and functional
constraints.

\vspace{5pt}

\noindent \textbf{Contributions}. These issues motivate our work
here.  In this paper we present a complete iteration complexity
analysis of inexact first order  Lagrangian and penalty methods for
solving cone constrained convex problems that have or may not have
optimal Lagrange multipliers that close the duality gap. In the
first part of our paper we assume the existence of optimal Lagrange
multipliers and we derive overall complexity  of primal-dual first
order methods based on the inexact oracle framework \cite{DevGli:14}
and augmented Lagrangian smoothing \cite{RocWet:98} or Nesterov type
smoothing \cite{BecTeb:12,Nes:051}. Although we obtain in some cases
similar complexity results with those found in the literature, our
analysis based on the inexact oracle framework is simpler, intuitive
and more elegant, opening various possibilities for extensions to
more complex optimization models. Moreover, in some  optimality
criteria our computational complexities are significantly better
than those found in the existing literature. These better
complexities are achieved through the new  first order inexact
oracles for augmented Lagrangian (Nesterov) smoothing derived in
Theorem 3.2 (Theorem 3.11) that improve substantially those in
\cite{DevGli:14}.
In the second part we assume that  the conic constrained convex
problem might not admit an optimal Lagrange multiplier. In this
case, we combine the fast gradient method with penalty strategies
and derive computational complexity certifications for such methods
which consistently improves those given in \cite{Nes:14} for the
nonsmooth case.  Thus, our results cover the particular case when
the Slater condition does not hold or it is difficult to check for
large-scale conic convex problems and answer positively to
Nesterov's question. To  the best of our knowledge, this paper
present one of the first computational complexity results for first
order penalty methods  for convex problems when optimal Lagrange
multipliers do not exist. More explicitly, our contributions are:

\vspace{5pt}

\noindent $(i)$  First, we assume that we have optimal Lagrange
multipliers that close the duality gap for the cone constrained
convex problem with simple or smooth objective function. We provide
new computational complexity results on the dual first order
augmented Lagrangian methods, where the main complexity bounds show
that, in order to obtain an $\epsilon-$optimal solution for the
original problem, the inexact (fast) gradient augmented Lagrangian
algorithms have to perform
$\mathcal{O}\left(\frac{1}{\epsilon}\right)$ total projections onto
the  simple primal  feasible set and  feasible cone.

\noindent $(ii)$ We combine in a novel fashion Nesterov smoothing technique
and inexact fast gradient method for solving  cone constrained
optimization problems with possibly unbounded feasible cone. We show
that, in order to obtain an $\epsilon$-optimal solution,  fast
gradient method with inexact information performs
$\mathcal{O}\left(\frac{1}{\epsilon^{3/2}} \log
\left(\frac{1}{\epsilon}\right)\right)$ projections onto the simple
primal feasible  set and $\mathcal{O}\left(\frac{1}{\epsilon}
\right)$ projections onto the feasible cone. Thus, our work shows
that inexact fast gradient method based on Nesterov smoothing has
worse overall complexity than the one based on augmented Lagrangian
smoothing.

\noindent $(iii)$ Then, we eliminate the assumption that there
exists some optimal Lagrange multiplier for the cone constrained
convex problem and we analyze the computational complexity of fast
gradient penalty methods. If the objective function is smooth, then
we prove that in order to obtain an $\epsilon$-optimal solution for
the original problem  we need to perform
$\mathcal{O}\left(\frac{1}{\epsilon^{3/2}}\right)$ total projections
onto the simple primal feasible set. Through an example, we also
show that our bounds are tight.

\vspace{5pt}

%

\noindent \textbf{Notations}. We denote $\bar{\rset} = \rset \cup \{
+ \infty \}$. For $u,v \in \rset^n$, we consider scalar product
$\langle u,v \rangle = u^T v$ and  Euclidean norm $\|u\|=\sqrt{u^T
u}$. Further, $[u]_U$ denotes the projection of $u$ onto
 {nonempty closed} convex set $U$ and
$\text{dist}_{U}(u) =\|u -[u]_U\|$ its distance to  $U$. Moreover,
we use notation $\mathcal{N}_{U}(u)$ for the normal cone of the
convex set $U$ at $u \in U$ defined by $\mathcal{N}_{U}(u) = \{t \in
\rset^n:\; \langle t, u - v \rangle \ge 0 \quad \forall v \in U \}$.
We also use notation $\mathcal{B}_r(x) = \{ z \in \rset^n| \;\;
\norm{z-x} \le r \}$. For a matrix $G \in \rset^{m \times n}$ we use
$\|G\|$ for the spectral norm.


\vspace{5pt}

\section{Problem formulation}
In this paper we consider the following  cone constrained  convex
optimization problem:
\begin{align}
\label{problem}
 f^* = \; & \min\limits_{u \in U} \quad f(u) \qquad  \text{s.t.}
 \quad Gu+g \in \mathcal{K},
\end{align}
where $f: \rset^n \to \bar{\rset}^{}$ is a proper, closed, convex
function, $U \subseteq \text{dom} f$ is a nonempty  closed, convex
set, $G \in \rset^{m \times n}$ and $\mathcal{K} \subseteq \rset^m$
is a nonempty, closed, convex cone, having its polar cone
$\mathcal{K}^* = \{ v \in \rset^m:\; \langle v, \kappa \rangle \le 0
\quad  \forall \kappa \in \mathcal{K} \}$. We denote $U^* \subseteq
\rset^n$  the optimal set of the above problem. Note that our
formulation and results can be extended to general normed vector
spaces. The following assumptions are valid throughout the paper:
\begin{assumption}
\label{strong_conv} Objective function $f$ is strongly convex with
constant $\sigma_f \ge 0$:

\vspace{-0.6cm}

\begin{equation*}
f \left(\alpha u + (1 - \alpha)v \right) \le \alpha f(u) +  (1 -
\alpha) f(v) - \frac{\sigma_f \alpha (1 - \alpha)}{2} \norm{u-v}^2
\;\; \forall u, v \!\in\! \text{dom} f, \; \alpha \in [0, \; 1].
\end{equation*}
\end{assumption}

\vspace{-0.4cm}

\noindent Note that Assumption \ref{strong_conv} with $\sigma_f=0$ is
equivalent with convexity of  function $f$.

\begin{assumption}\label{simple_set}
\noindent $(i)$ \; The feasible set $U$ and the cones  $\mathcal{K}$
and $\mathcal{K}^*$ are closed, convex and simple (e.g., the
projection onto these sets can be
obtained in closed form). \\
\noindent $(ii)$ The convex set $U$ is bounded, i.e. exists $D_U <
\infty$ such that $\max\limits_{u,v \in U} \norm{u-v} \le D_U$.
\end{assumption}

\noindent Note that these assumptions are standard in the context of
first order Lagrangian and penalty methods for conic convex
programming, see e.g.
\cite{AybIye:13,LanMon:13,LanMon:15,Nem:04,Nes:051,QuoCev:14}.
Further, convex function $h: \rset^n \to \bar{\rset}^{}$, with $U
\subseteq \text{dom} h$, is called \textit{simple} if the optimal
solution of the following problem can be efficiently obtained (e.g.,
in closed form):

\vspace{-0.5cm}

$$\min\limits_{u \in U} \; h(u) + \frac{1}{2\mu}\norm{u-z}^2
\qquad \forall \mu>0 \;\; \text{and} \;\; z \in \rset^n.$$

\vspace{-0.3cm}

\noindent In this paper we assume that the convex objective function
$f$ is either simple or has Lipschitz continuous gradient with
constant $L_f>0$ and $\text{dom} f = \rset^n$, i.e.:

\vspace{-0.5cm}

\[ 0 \leq f(y) - \left( f(x) + \langle \nabla f(x), y -x \rangle  \right)
 \leq \frac{L_f}{2} \| x - y\|^2 \quad \forall x,y \in \rset^n. \]

\vspace{-0.3cm}

\noindent Our goal is to  find an approximate  solution for the
optimization problem \eqref{problem}. Thus,   we introduce the
following  definition used in the rest of the paper:
\begin{definition}\label{opt_point}
Given the desired accuracy $\epsilon>0$, the primal  point
${u}_\epsilon \in U$ is an \textit{$\epsilon$-optimal solution} for
the cone constrained convex problem \eqref{problem} if it satisfies:
$$|f({u}_\epsilon ) - f^*| \le \epsilon \quad \text{and} \quad \text{dist}_{\mathcal{K}}(G{u}_\epsilon +g) \le \epsilon.$$
\end{definition}


\subsection{A framework for inexact first order methods}
\noindent Since  the main algorithm in this paper  is the Nesterov
fast gradient method \cite{Nes:13}, we further introduce an inexact
algorithmic framework based on the method in
\cite{BecTeb:09,Tse:08}, which will be subsequently called in
various ways. Therefore, consider the following general convex
constrained optimization problem with composite objective function:
\begin{equation}
\label{general_comp} F^* =  \min\limits_{z \in Q} \; F(z) \qquad
\left(= \phi(z) + \psi(z) \right),
\end{equation}
where $Q \subseteq \rset^n $ is a simple, convex set, $\phi:\rset^n
\to \rset^{}$ is a  convex function with Lipschitz continuous
gradient  of constant $L_\phi >0$ and $\psi:\rset^n \to
\bar{\rset}^{}$ is a simple, closed, convex function. Using the
definition from \cite{DevGli:14}, given $\delta \geq 0$ and $L>0$,
we assume that the smooth function $\phi$ is equipped with a
\textit{first order inexact $(\delta,L)$-oracle}, i.e. for any $y
\in Q$ we can  compute an approximate function value
$\phi_{\delta,L}(y)$ and an approximate gradient $\nabla
\phi_{\delta,L}(y)$ such that the following inequalities hold:
\begin{equation}
\label{deltaL} 0 \le \phi(x) - \left(\phi_{\delta,L}(y) + \langle
\nabla \phi_{\delta,L}(y), x-y\rangle \right) \le
\frac{L}{2}\norm{x-y}^2 + \delta \qquad \forall x \in Q.
\end{equation}

\noindent Next, we introduce the Inexact Composite Fast Gradient
(ICFG) method for solving the composite optimization problem
\eqref{general_comp} using approximate function values and gradients
$(\phi_{\delta,L}(y), \nabla \phi_{\delta,L}(y))$ satisfying the
first order  $(\delta,L)$-oracle given in \eqref{deltaL}:

\vspace{3pt}

\begin{center}
\framebox{
\parbox{11.5cm}{
\noindent \textbf{ Algorithm ICFG ($\phi, \psi, \delta, L$) }\\
Give $z^0 = w^1 \in \rset^n$ and $\theta_1=1$. For $k\geq 1$ do:
\begin{enumerate}
\item Compute the pair $(\phi_{\delta,L}(w^k),\nabla \phi_{\delta,L}(w^k))$ satisfying \eqref{deltaL}. Update:
\item ${z}^{k}= \arg\min\limits_{z \in Q} \langle \nabla \phi_{\delta,L}(w^k),z - w^k\rangle + \frac{L}{2}\norm{z-w^k}^2 + \psi(z)$
\item $w^{k+1} = z^k + \frac{\theta_k -1 }{\theta_{k+1}} (z^k - z^{k-1})$
\item If a stopping criterion holds, then STOP and \textbf{return}: $(z^k, w^k)$.
\end{enumerate}
}}
\end{center}

\vspace{3pt}

\noindent Note that if we update $\theta_{k+1} = \frac{1 +  \sqrt{1
+ 4 \theta_k^2 }}{2}$ for all $k \geq 1$ and additionally we
consider  $\delta=0$ and $L=L_\phi$, then we recover the well-known FISTA scheme which has been analyzed for the first time
in \cite{BecTeb:09} and subsequently extended in different variants in \cite{Tse:08,Nes:13}.
On the other hand, if
we take $\theta_k = 1$ for all $k \ge 1$ and $\delta=0$, then $z^{k}
= w^{k+1}$ and we recover the ISTA scheme also developed in \cite{BecTeb:09}
and extended in \cite{NedOzd:09,NecNed:13,Nes:13}.
Using the
same reasoning as in \cite{DevGli:14}, we provide in the next
theorem the rate of convergence of Algorithm \textbf{ICFG} for
composite optimization problem \eqref{general_comp} endowed  with a
first order inexact $(\delta,L)$-oracle \eqref{deltaL}. First, let us
denote by $z^*$ an optimal solution of the composite convex  problem
\eqref{general_comp}.

\begin{theorem}\cite{DevGli:14,BecTeb:09}
\label{ifg_rate_conv} Let sequences $(z^k, w^k)_{k \ge 0}$ be
generated by Algorithm \textbf{ICFG} ($\phi,\psi,\delta,L$) for
solving composite  problem \eqref{general_comp} endowed with a first
order inexact $(\delta,L)$-oracle. Then, we have
the following convergence rates in terms of function values:\\
\noindent (i) If we define the average sequence
$\hat{z}^k=\frac{1}{k} \sum\limits_{i=0}^{k-1} z^{i+1}$ and
$\theta_k = 1$ for all $k \ge 1$, then  $\hat{z}^k$ has the
following sublinear convergence rate in terms of function values:
\begin{equation*}
 F(\hat{z}^k) -  F^* \le \frac{ L \norm{z^0 -z^*}^2}{2k} + \delta.
\end{equation*}
\noindent (ii) If we update  $\theta_{k+1} = \frac{1 + \sqrt{1 + 4
\theta_k^2 }}{2}$ for all $k \ge 1$, then the last iterate $z^k$ has
the following sublinear convergence rate in terms of function
values:
\begin{equation*}
 F(z^k) - F^* \le \frac{2 L \norm{z^0 -z^*}^2}{(k+1)^2} + k \delta.
\end{equation*}
\end{theorem}



\section{Inexact first order Lagrangian methods}
\label{sec_Lagrangian} In this section we analyze the computational
complexity of inexact first order Lagrangian methods for solving the
cone constrained  convex problem \eqref{problem}.  Since we use the
dual framework, we require  the following standard assumption for
dual algorithms, valid only in this Section \ref{sec_Lagrangian}:
\begin{assumption}\label{assump_bound_multipliers}
 {There exists a Lagrange multiplier $x^* \in
\mathcal{K}^*$  for the conic convex problem \eqref{problem} that
closes the duality gap.}
\end{assumption}

\noindent Assumption \ref{assump_bound_multipliers} implies the
existence of a bounded optimal Lagrange multiplier, that is
$\norm{x^*} < \infty$, and it holds for \eqref{problem} whenever a
Slater type condition is valid, i.e. there exists $\bar{u} \in
\text{relint}(U)$ such that $G \bar{u} + g \in
\text{relint}(\mathcal{K})$.


\subsection{Preliminaries}
\noindent The strongly convex case,  i.e. when the objective
function $f$ in problem \eqref{problem} satisfies Assumption
\ref{strong_conv} with $\sigma_f>0$, has been extensively studied in
the literature, see e.g.
\cite{NecPat:14,QuoCev:14,NecNed:13,YurQuo:15}. Thus, in the rest of
our paper, unless it is explicitly stated, we assume that the
function $f$ is convex, i.e. it satisfies Assumption
\ref{strong_conv} with $\sigma_f=0$. In the general convex case, the
dual function, denoted $d$, is nonsmooth, and thus  dual first order
methods,  such as Algorithm \textbf{ICFG}, cannot be applied. In
order to be able to apply dual first order algorithms, our approach
relies on the combination between smoothing techniques and duality.
First, we introduce some notations. We note that the problem
\eqref{problem} can be reformulated equivalently as:
\begin{equation}
\label{problem2} \min_{u,s} \; f(u)  \quad \text{s.t.} \;\;  u \in
U, \; s \in \mathcal{K}, \; Gu+g=s.
\end{equation}
Thus, the Lagrangian and the dual function of the convex problem
\eqref{problem2}   are given  by:
\begin{equation*}
  \mathcal{L}(u,s,x) = f(u) + \langle x, Gu+g-s \rangle \quad \text{and} \quad  d(x) = \min\limits_{u \in U, s \in \mathcal{K}} \; \mathcal{L}(u,s,x).
\end{equation*}
 {Assumption \ref{assump_bound_multipliers} states
that there exists a Lagrange multiplier  $x^* \in \mathcal{K}^*$
such that $f^* = d(x^*)$ and thus  the  convex problem
\eqref{problem} is equivalent with solving the dual formulation}:
\begin{align}
\label{classic dual} f^* = \max_{x \in \rset^m} d(x).
\end{align}
We denote with $X^*$ the set of optimal solutions of the dual
problem \eqref{classic dual}. Various dual subgradient schemes have
been developed for solving \eqref{classic dual} with $\epsilon$
accuracy, with overall complexities of  order
$\mathcal{O}\left(\frac{1}{\epsilon^2}\right)$
\cite{NedOzd:09,Nes:14}. However, under additional mild assumptions,
we aim in this paper at improving the iteration complexity required
for solving the conic optimization problem \eqref{problem} using the
dual formulation. First, we rewrite the dual function $d$ in a novel
way as a composite function:
\begin{equation}
\label{dual_separation} d(x) = \underbrace{\min\limits_{u \in U} \;
[f(u) + \langle x, Gu+g \rangle ] }_{d_U(x)} \;+\;
\underbrace{\min\limits_{s \in \mathcal{K}} \; \langle
-s,x\rangle}_{d_{\mathcal{K}}(x)} = d_U(x) + d_{\mathcal{K}}(x).
\end{equation}
The function $d_{\mathcal{K}}(x)$ is the support  function of the
cone $\mathcal{K}$ and, by the definition of the polar cone
$\mathcal{K}^*$, also represents the indicator function of
$\mathcal{K}^*$. From our knowledge there are two widely known
smoothing strategies to obtain an approximate dual function with
Lipschitz continuous gradient. They  are based on the following
modified Lagrangian and
dual functions:\\
\noindent$(i)$ Augmented Lagrangian smoothing  \cite{AybIye:13,
LanMon:15,NedNec:14,QuoCev:14}:
\begin{align*}
  \mathcal{L}^{\text{ag}}_{\mu} (u,s,x) &= f(u) + \langle x, Gu+g-s \rangle + \frac{\mu}{2}\norm{Gu+g-s}^2 \\
 d^{\text{ag}}_{\mu} (x) &= \min\limits_{u \in U, s \in \mathcal{K}} \; \mathcal{L}^{\text{ag}}_{\mu}(u,s,x).
\end{align*}
\noindent $(ii)$ Nesterov smoothing
\cite{BotHen:15,BotHen:15o,NecSuy:08,Nes:051,QuoNec:15,YurQuo:15}:
\begin{align*}
  \mathcal{L}^{\text{ns}}_{\mu}(u,s,x) &= f(u) + \langle x, Gu+g-s \rangle + \frac{\mu}{2}(\norm{u}^2 + \norm{s}^2) \\
 d^{\text{ns}}_{\mu}(x) &= \min\limits_{u \in U, s \in \mathcal{K}} \; \mathcal{L}^{\text{ns}}_{\mu}(u,s,x).
\end{align*}
Note that, following the reasoning from
\cite{Nes:051,NecSuy:08,QuoNec:15}, the Nesterov smoothing
approximation $d^{\text{ns}}_{\mu}(x)$ requires the boundedness of
the primal feasible set $\mathcal{K} \times U$. Thus, the general
convex cone $\mathcal{K}$  induces difficulties in using this
strategy. In Section \ref{subsec_nssmoothing} we present a modified
Nesterov smoothing technique which is able to cope with linear conic
constraints and unbounded feasible cone $\mathcal{K}$ based on the
new composite reformulation  \eqref{dual_separation}.


\subsection{Inexact first order methods for augmented Lagrangian smoothing}
In this section, we analyze the iteration complexity of the inexact
first order methods for augmented Lagrangian smoothing, under
Assumption \ref{strong_conv} with $\sigma_f = 0$, Assumption
\ref{simple_set} and Assumption \ref{assump_bound_multipliers}. The
inexact gradient Lagrangian method is equivalent with the classical
augmented Lagrangian algorithm, namely the application of the
inexact gradient method on the augmented dual function. The second
first order Lagrangian method we analyze is the inexact fast
gradient Lagrangian method, which is based on the application of the
fast gradient method on the augmented  dual function.  We start with
the classical augmented Lagrangian setting, i.e. we define
\cite{RocWet:98}:
\begin{align*}
 \mathcal{L}^{\text{ag}}_{\mu} (u,x) \!=\! f(u) +
\frac{\mu}{2}\text{dist}_{\mathcal{K}} \!\left( \!
Gu+g+\frac{1}{\mu}x \! \right)^2 \!- \frac{1}{2\mu}\norm{x}^2 \;\;
\text{and} \;\;
 d^{\text{ag}}_{\mu} (x) \!=\! \min\limits_{u \in U} \; \mathcal{L}^{\text{ag}}_{\mu}(u,x).
\end{align*}
\noindent Note that, the augmented dual function represents a pure
Moreau approximation of  the original dual function:
\begin{equation*}
d^{\text{ag}}_{\mu} (x) = \max\limits_{z \in \rset^m} \;  d(z) -
\frac{1}{2\mu}\norm{z-x}^2 = \max\limits_{z \in \mathcal{K}^*} \;
d_U(z) - \frac{1}{2\mu}\norm{z-x}^2.
\end{equation*}
Further, we observe that partial gradient of
$\mathcal{L}_{\mu}^{\text{ag}}$  w.r.t. $x$  is given by:
\begin{align*}
\nabla_x \mathcal{L}^{\text{ag}}_{\mu}(u,x) &= Gu + g  - \left[Gu+g
+ \frac{1}{\mu}x \right]_{\mathcal{K}}.
\end{align*}
For any $x \in \rset^m$ we denote a primal exact solution by
$u_{\mu}^*(x) \in \arg\min\limits_{u \in U}
\mathcal{L}_{\mu}^{\text{ag}}(u,x)$. It is well-known, see e.g.
\cite{RocWet:98}, that the gradient of augmented dual function
$d^{\text{ag}}_{\mu}(x)$ satisfies:
$$\nabla d^{\text{ag}}_{\mu}(x) = Gu_{\mu}^*(x) + g  -
\left[Gu_{\mu}^*(x) + g + \frac{1}{\mu}x  \right]_{\mathcal{K}},$$
and  additionally it is Lipschitz continuous  with constant
$L_{\text{d}} = \frac{1}{\mu}$. Moreover,   the resulting augmented
dual problem, given by:
\begin{equation*}
f^* =  \max\limits_{x \in \rset^m} \; \; d^{\text{ag}}_{\mu}(x),
\quad \text{satisfies} \quad X^* = \arg\max\limits_{x \in \rset^m}
\; d^{\text{ag}}_{\mu}(x).
\end{equation*}
Usually, it is difficult to compute in most of the practical
applications
 the optimal solution $u_{\mu}^*(x)$ of the inner problem $\min\limits_{u \in U}
\mathcal{L}_{\mu}^{\text{ag}}(u,x)$ and we can obtain only an
approximate solution. Assume that we solve inexactly the inner
problem and obtain an approximate solution $u_{\mu}(x) \in U$ which,
for a given accuracy $\delta>0$, satisfies:
\begin{equation}\label{func_approx2}
 0 \le  \mathcal{L}^{\text{ag}}_{\mu}(u_{\mu}(x),x) - d^{\text{ag}}_{\mu}(x) \le
  \delta \quad \forall x \in \rset^m.
\end{equation}
Then, we can construct a first order inexact oracle for the
augmented dual function:

\begin{theorem}\label{th_inexact_oracle}
Let $\mu,\delta>0$, then we have the following first order inexact
 $(3\delta, 2L_{\text{d}})$-oracle for the augmented dual function $d^{\text{ag}}_{\mu}$:
\begin{equation}\label{inexact_descentlemma}
0 \le \mathcal{L}^{\text{ag}}_{\mu}(u_{\mu}(y),y) + \langle \nabla_x
\mathcal{L}^{\text{ag}}_{\mu}(u_{\mu}(y),y), x-y\rangle
-d^{\text{ag}}_{\mu}(x) \le  \frac{2L_{\text{d}}}{2}\norm{x-y}^2 +
 {3\delta},
\end{equation}
for all $x,y \in \rset^m$, where the approximate solution
$u_{\mu}(y)$ satisfies \eqref{func_approx2} and $L_d =
\frac{1}{\mu}$.
\end{theorem}

\begin{proof}
For the left hand side  inequality of \eqref{inexact_descentlemma},
we observe that:
\begin{align*}
\mathcal{L}_{\mu}^{\text{ag}}(u_{\mu}(y),y) +  \langle \nabla_x
\mathcal{L}_{\mu}^{\text{ag}} (u_{\mu}(y),y), x - y \rangle & \geq
\mathcal{L}_{\mu}^{\text{ag}}(u_{\mu}(y),x) \ge \min\limits_{u \in
U} \; \mathcal{L}_{\mu}^{\text{ag}}(u,x)  = d_{\mu}^{\text{ag}}(x).
\end{align*}
For the right hand side inequality of \eqref{inexact_descentlemma},
note  that for any fixed $u \in U$ the function $h(x) =
\mathcal{L}_{\mu}^{\text{ag}}(u,x) - d_{\mu}^{\text{ag}}(x)$ has
Lipschitz gradient with constant  {$L_{\text{h}} =
2/\mu$} and $h(x) \ge 0$ for all $x \in \rset^m$. Therefore, using
the notation $L_{\text{d}} = 1/\mu$, we have:
\begin{align*}
h(x) - \min_{x \in \rset^m} h(x) \ge
\frac{1}{ {2L_{\text{h}}}} \norm{\nabla h(x)}^2  =
\frac{1}{ {4L_{\text{d}}}} \norm{\nabla_x
\mathcal{L}_{\mu}^{\text{ag}}(u,x) - \nabla
d_{\mu}^{\text{ag}}(x)}^2 \qquad \forall u \in U.
\end{align*}
Taking $u =  u_{\mu}(x)$ and using the definition of $u_{\mu}(x)$,
we have $h(x) - \min_{x \in \rset^m} h(x) \leq h(x) \leq \delta$ and
obtain the following approximate gradient relation:
\begin{equation}\label{grad_approx}
\norm{\nabla_x \mathcal{L}_{\mu}^{\text{ag}}(u_{\mu}(x),x) - \nabla
d_{\mu}^{\text{ag}}(x)}  = \norm{Gu_{\mu}(x)  - Gu_{\mu}^*(x)}\le
\sqrt{ {4\delta L_{\text{d}}}}.
\end{equation}
From the Lipschitz continuity of $\nabla d_{\mu}^{\text{ag}}$,
\eqref{func_approx2} and \eqref{grad_approx}, we have that for any
$x,y \in \rset^m$ the following relations hold:
\begin{align*}
& d_{\mu}^{\text{ag}}(x) \ge d_{\mu}^{\text{ag}}(y) + \langle \nabla d_{\mu}^{\text{ag}}(y), x - y\rangle - \frac{L_{\text{d}}}{2}\norm{x - y}^2 \\
 & \overset{\eqref{func_approx2}}{\ge} \mathcal{L}_{\mu}^{\text{ag}}(u_{\mu}(y),y) + \langle \nabla d_{\mu}^{\text{ag}}(y), x - y\rangle - \frac{L_{\text{d}}}{2}\norm{x - y}^2 -\delta\\
 & = \mathcal{L}_{\mu}^{\text{ag}}(u_{\mu}(y),y) + \langle \nabla_x \mathcal{L}_{\mu}^{\text{ag}}(u_{\mu}(y),y), x - y \rangle - \frac{L_{\text{d}}}{2}\norm{x - y}^2 -\delta\\
 & \qquad \qquad \qquad    + \langle \nabla d_{\mu}^{\text{ag}}(y) - \nabla_x \mathcal{L}_{\mu}^{\text{ag}}(u_{\mu}(y),y), x - y \rangle \\
 & \ge \mathcal{L}_{\mu}^{\text{ag}}(u_{\mu}(y),y) + \langle \nabla_x \mathcal{L}_{\mu}^{\text{ag}}(u_{\mu}(y),y), x - y \rangle - \frac{L_{\text{d}}}{2}\norm{x - y}^2 -\delta\\
 & \qquad \qquad \qquad    - \norm{\nabla d_{\mu}^{\text{ag}}(y) - \nabla_x \mathcal{L}_{\mu}^{\text{ag}}(u_{\mu}(y),y)} \norm{x - y} \\
 & \overset{\eqref{grad_approx}}{\ge} \mathcal{L}_{\mu}^{\text{ag}}(u_{\mu}(y),y) + \langle \nabla_x \mathcal{L}_{\mu}^{\text{ag}}(u_{\mu}(y),y), x - y \rangle - \frac{L_{\text{d}}}{2}\norm{x - y}^2 -\delta
- \sqrt{ {4\delta L_{\text{d}}}} \norm{x - y}.
\end{align*}
On the other hand, for any positive pair of constants $(t, \alpha)$
we have:  {$ \alpha t \le \frac{t^2}{2} +
\frac{\alpha^2}{2}$}. Thus, taking  {$t =
\sqrt{L_{\text{d}}}\norm{x - y}$ and $\alpha = 2\sqrt{\delta} $} in
the previous inequalities, we obtain the right hand side inequality
of the theorem.
\end{proof}

\noindent The relation \eqref{inexact_descentlemma} implies that the
augmented dual function $d_{\mu}^{\text{ag}}$ is  smooth and  is
equipped with a first order inexact $(3\delta,
2L_{\text{d}})$-oracle having $\phi_{\delta,L} (x) =
\mathcal{L}_{\mu}^{\text{ag}}(u_{\mu}(x),x)$ and $\nabla
\phi_{\delta,L} (x)= \nabla_x
\mathcal{L}_{\mu}^{\text{ag}}(u_{\mu}(x),x) = Gu_{\mu}(x)+g$. It is
important to note that many previous results on augmented Langragian
methods require solving  the inner problem with much higher inner
accuracy of order $\mathcal{O}(\delta^2)$ (see e.g.
\cite{AybIye:13,DevGli:14,LanMon:15,NedNec:14,QuoCev:14}), i.e.:
$$ \mathcal{L}_{\mu}^{\text{ag}}(u_{\mu}(x),x) - d_{\mu}^{\text{ag}}(x) \le
\mathcal{O}(\delta^2).$$ It is obvious that our approach here is
less conservative, imposing to solve the inner problem with less
inner accuracy of order $\delta$ as in \eqref{func_approx2}. As we
will see in the sequel, this new and important  result from Theorem
\ref{th_inexact_oracle} will have a huge impact on the
computational complexity of our methods compared to those given in
the previous papers. In particular, the first order inexact oracle
derived in \cite{DevGli:14} for augmented Lagrangian dual function
is more conservative than the one from Theorem
\ref{th_inexact_oracle} and thus, its direct application will lead
to much worse computational complexities than the ones we obtained
in the present paper based on Theorem \ref{th_inexact_oracle}.

\noindent Given the pair $(x^k,y^k)_{k \ge 0}$ generated by
Algorithm \textbf{ICFG}, in the following two sections we provide
complexity estimates related to the convergence of the average
primal point $({\hat u}^k)_{k \geq 0}$ defined in a compact way as
follows:
\begin{equation}\label{average_seq}
\hat{u}^k =
\frac{1}{S^{\theta}_k}\sum\limits_{i=0}^{ {k-1}}
\theta_i u^i, \quad \text{where} \quad  S^{\theta}_{k} =
\sum\limits_{i=0}^{ {k-1}} \theta_k \;\;\; \text{and}
\;\;\; u^i = u_{\mu}(x^i).
\end{equation}
 {Notice that the weights  $\theta_k$ are either
constant, i.e. $\theta_k =1$ for all $k \ge 0$, or updated as
$\theta_0=0$, $\theta_1=1$ and $\theta_{k+1}  = \frac{1+\sqrt{1 + 4
\theta_k^2}}{2}$ for all $k \ge 1$.}



\subsubsection{Inexact gradient augmented Lagrangian  method}
We now analyze the overall complexity of the classical augmented
Lagrangian method in terms of projections onto the cone
$\mathcal{K}$ and simple feasible set $U$, under various assumptions
on the objective function $f$. A direct consequence of Theorem
\ref{th_inexact_oracle}  and Theorem \ref{ifg_rate_conv} is the
following iteration complexity (in terms of outer iterations) of the
inexact gradient augmented Lagrangian method.

\begin{corollary}\label{corr_icfg_grad}
Under Assumptions \ref{strong_conv} with $\sigma_f = 0$,
\ref{simple_set} and \ref{assump_bound_multipliers}, let $\mu,
\delta > 0$ and $(x^k)_{k \ge 0}$ be the sequence generated by
Algorithm \textbf{ICFG} $(d^{\text{ag}}_{\mu},0,3\delta,
2L_{\text{d}})$ with $\theta_k=1$ for all $k \ge 0$ and
$L_{\text{d}} = \frac{1}{\mu}$.  Define the average sequence
$(\hat{x}^k)_{k \ge 1}$ by $\hat{x}^k =
\frac{1}{k}\sum\limits_{i=0}^{k-1} x^{i+1} $. Then,  we have the
following convergence  estimate on dual suboptimality:
\begin{equation*}
 f^* - d_{\mu}^{\text{ag}}(\hat{x}^k) \le \frac{L_{\text{d}} R_{\text{d}}^2}{k} + 3 \delta.
\end{equation*}
\end{corollary}
\noindent Note that the above convergence rate is linked only  to
the number of the outer iterations and omits the complexity of
solving the inner subproblem at step $1$ of \textbf{ICFG}. Before
estimating the total complexity of the process containing the inner
and outer levels, we provide convergence rates for the primal
infeasibility and suboptimality.

\begin{theorem}\label{aug_grad_outer}
Under  Assumptions \ref{strong_conv} with $\sigma_f=0$,
\ref{simple_set} and \ref{assump_bound_multipliers}, let
$\mu,\delta>0$ and $(x^k)_{k \ge 0}$ be the sequence generated by
Algorithm
\textbf{ICFG}$(d^{\text{ag}}_{\mu},0,3\delta,2L_{\text{d}})$ with
$\theta_k = 1$ for all $k \ge 0$ and $L_{\text{d}} = \frac{1}{\mu}$.
Let $u^i = u_{\mu}(x^i)$ be such that
$\mathcal{L}_{\mu}^{\text{ag}}(u^i,x^i) - d_{\mu}^{\text{ag}}(x^i)
\le  \delta$ for $0 \le i \le k$. Then, the average primal sequence
$ ({\hat u}^k)_{k \geq 1}$ defined by \eqref{average_seq} satisfies
the following relations:\\
$(i)$ The primal infeasibility is bounded sublinearly as
 follows:
\begin{equation*}
\text{dist}_{\mathcal{K}}(G\hat{u}^k+g) \le
\frac{4L_{\text{d}}R_{\text{d}}}{k} +
\sqrt{\frac{12L_{\text{d}}\delta}{k}}.
\end{equation*}
$(ii)$ The primal suboptimality gap is bounded by:
\begin{equation*} 
- \frac{4L_{\text{d}}R_{\text{d}}^2}{k} - R_{\text{d}}
\sqrt{\frac{12L_{\text{d}}\delta}{k}} \le f(\hat{u}^k) -f^* \le
\frac{L_{\text{d}}\norm{x^0}^2}{k} + 3\delta.
\end{equation*}
\end{theorem}

\begin{proof}
In order to facilitate an easy reading of the results,  we provide
the proof of  primal infeasibility and suboptimality bounds in
Appendix A.1.
\end{proof}

\noindent Note that using the above rate of convergence, one might
choose an appropriate value of the smoothing parameter $\mu$ such
that after a single outer iteration an $\epsilon-$optimal point is
obtained. Thus, convergence rates in terms of outer iterations are
not relevant in this case and it is natural to analyze the
computational complexity of the Algorithm \textbf{ICFG}, by taking
into account also the complexity of solving the inner subproblems.
 { Therefore, we need  to also count the number of
fast gradient steps, which includes projections onto  $U$ and
$\mathcal{K}$, matrix-vector products $Gu$ and $G^Tx$, or gradient
computations $\nabla f(u)$, performed in order to attain the
required inner accuracy, at a given outer iteration. Since,  in the
literature this is usually measured in terms of projections onto $U$
and $\mathcal{K}$ (see e.g. \cite{AybIye:13,LanLu:11,LanMon:15}), we
also use this measure of computational complexity.} We further
analyze the necessary number of inner projections that the inexact
gradient Lagrangian method has to perform at each outer iteration. A
well-known fact that we use further is that the function $u \mapsto
\text{dist}_{\mathcal{K}} \left(G u + g \right)^2$ has Lipschitz
continuous gradient with constant $\norm{G}^2$ \cite{LanLu:11}.
Using this observation, depending of the assumptions on the function
$f$, we have the following inner iteration complexities for solving
approximately the inner problem $\min_{u \in U}
\mathcal{L}^{\text{ag}}_{\mu} (u,x)$ for a given
$x$:\\

\noindent $(i)$ If the function $f$ is simple, then  Algorithm
\textbf{ICFG}($\frac{\mu}{2}\text{dist}_{\mathcal{K}}(G\cdot + g +
\frac{1}{\mu}x)^2,f,0,\mu\norm{G}^2$) returns a primal point
$u_{\mu}(x) \in U$ such that $\mathcal{L}^{\text{ag}}_{\mu}
(u_{\mu}(x),x) - d^{\text{ag}}_{\mu}(x)\le \delta$ after:
\begin{equation*}
N_{\delta}^{\text{in}}  = \left \lceil
\norm{G}D_U\sqrt{\frac{2\mu}{\delta}} \right\rceil
\end{equation*}
projections onto the primal simple feasible set $\mathcal{K} \times
U$.

\noindent $(ii)$ If the function $f$ is not simple, but its gradient
$\nabla f$ is Lipschitz continuous with constant $L_f>0$, then the
Algorithm \textbf{ICFG}($ \mathcal{L}^{\text{ag}}_{\mu} (\cdot, x),
0, 0,L_f + \mu\norm{G}^2$) returns a primal point $u_{\mu}(x) \in U$
such that $\mathcal{L}^{\text{ag}}_{\mu}(u_{\mu}(x),x) -
d^{\text{ag}}_{\mu}(x)\le \delta$ after:
\begin{equation}\label{inner_comp_Lips}
N_{\delta}^{\text{in}} =  \left \lceil D_U\sqrt{\frac{2(L_f + \mu
\norm{G}^2)}{\delta}} \right \rceil
 \end{equation}
projections onto the primal simple feasible set $\mathcal{K} \times U$.\\

\noindent Note that if we take $L_f = 0$ in the iteration complexity
\eqref{inner_comp_Lips}, we recover the convergence rate  for the
case when $f$  is simple function. Therefore, for a uniform
complexity analysis, we provide in the following result an upper
bound on the total number of projections (for an optimal smoothing
parameter $\mu$) performed by the Algorithm \textbf{ICFG}, which  is
dependent on $L_f$ in the following sense: with some abuse of
notation for $f$ simple function we make the convention that  $L_f =
0$, and thus  we obtain the computational  complexity for simple
functions; otherwise, if we consider $L_f>0$, then we recover the
overall  complexity for the case when $\nabla f$ is $L_f$-Lipschitz
continuous. Moreover, we assume for simplicity that $x^0=0$.

\begin{theorem}\label{th_complexity_aug}
Under Assumptions \ref{strong_conv} with $\sigma_f=0$,
\ref{simple_set} and \ref{assump_bound_multipliers}, let $\mu,
\epsilon,\delta >0$ and $ (x^k)_{k \ge 0}$ be generated by Algorithm
\textbf{ICFG}$(d^{\text{ag}}_{\mu},0,3\delta,2L_{\text{d}})$ with
$\theta_k = 1$ for all $k \ge 0$. Assume that at each outer
iteration $k$, Algorithm
\textbf{ICFG}($\frac{\mu}{2}\text{dist}_{\mathcal{K}}(G\cdot + g +
\frac{1}{\mu}x^k)^2,f,0,\mu\norm{G}^2$) (if $f$ is simple and with
some abuse of notation we make the convention that $L_f=0$) or
Algorithm \textbf{ICFG}($ \mathcal{L}^{\text{ag}}_{\mu} (\cdot,
x^k), 0, 0,L_f + \mu\norm{G}^2$) (if $\nabla f$ is $L_f>0$ Lipschitz
continuous) is called to solve the inner problem and obtain a primal
approximate solution $u^k = u_\mu(x^k)$ such that
$\mathcal{L}^{\text{ag}}_{\mu}(u^k,x^k) -
d^{\text{ag}}_{\mu}(x^k)\le \delta$. Then, by setting the optimal
smoothing parameter:
\begin{equation}\label{smpar_aug_Lips}
\mu = \max\left\{\frac{16R_{\text{d}}^2}{\epsilon},
\frac{L_f}{\norm{G}^2}\right\}   \qquad \text{and} \qquad \delta =
\frac{\epsilon}{3}
\end{equation}
the average primal point $\hat{u}^k$ defined by \eqref{average_seq}
is $\epsilon-$optimal after a total number of
\begin{equation*}
 k = \left \lceil \sqrt{\frac{24 L_fD_U ^2}{\epsilon}} + \frac{6 \norm{G} D_U R_{\text{d}}}{\epsilon} \right \rceil
\end{equation*}
projections onto the primal simple feasible set $\mathcal{K} \times
U$.
\end{theorem}

\begin{proof}
Using the inner accuracy from \eqref{smpar_aug_Lips}  into Theorem
\ref{aug_grad_outer}, then the outer iteration complexity of the
augmented Lagrangian  method is given by:
\begin{equation}\label{gradaug_outer_estimate}
N_{\epsilon}^{\text{out}} = \left \lceil \frac{16L_d
R_d^2}{\epsilon} \right \rceil = \left \lceil \frac{16
R_d^2}{\mu\epsilon} \right \rceil.
\end{equation}
Based on the general inner complexity \eqref{inner_comp_Lips}, we
are able to tackle both cases:   when $f$ is simple or $\nabla f$ is
$L_f>0$ Lipschitz continuous. Minimizing the upper bound on the
product $N_{\epsilon}^{\text{out}} N^{\text{in}}_{\delta}$ over
positive parameters $\mu$ we get that the value of $\mu$ given in
\eqref{smpar_aug_Lips} is optimal up to a constant w.r.t. the total
complexity. Combining  \eqref{smpar_aug_Lips}, i.e. $\mu =
\max\left\{\frac{16R_d^2}{\epsilon},
\frac{L_f}{\norm{G}^2}\right\}$,  with \eqref{inner_comp_Lips}, we
obtain the following bound on the overall  complexity:
\begin{align*}
 N_{\epsilon}^{\text{out}} N_{\delta}^{\text{in}}
 & \le  \left\lceil \frac{16 R_{\text{d}}^2}{\mu \epsilon} \right\rceil
 \left(D_U\sqrt{\frac{6(L_f + \mu \norm{G}^2)}{\epsilon}} + 1\right) \nonumber \\
 & \le  \left(\frac{16 R_{\text{d}}^2}{\mu \epsilon} + 1 \right)
 2 D_U\sqrt{\frac{6(L_f + \mu \norm{G}^2)}{\epsilon}}  \nonumber \\
 & \overset{\eqref{smpar_aug_Lips}}{ = } 2 D_U \sqrt{\frac{6L_f}{\epsilon} +
  \frac{6 \norm{G}^2 R_{\text{d}}^2}{\epsilon^2}} + 1
   \le  \sqrt{\frac{24 L_fD_U}{\epsilon}} + \frac{6D_U \norm{G} R_{\text{d}}}{\epsilon} + 1.
 \end{align*}
Note that if we take $L_f = 0$ we obtain an upper bound on the
overall complexity of Algorithm \textbf{ICFG} for the case when $f$
is a  simple function.
\end{proof}

\begin{remark}
It is interesting to observe that choosing the smoothing parameter
$\mu \ge \frac{16R_{\text{d}}^2}{\epsilon}$, the
estimate \eqref{gradaug_outer_estimate} leads to the fact that the
inexact gradient augmented Lagrangian method terminates after
solving only once the inner subproblem. In other words, if an upper
bound on $R_{\text{d}}$ is known, then starting from an arbitrary
dual initial point $x^0 \in \rset^m$, it is sufficient to compute
$u_{\mu}(x^0) \in U$ satisfying
$\mathcal{L}^{\text{ag}}_{\mu}(u_{\mu}(x^0),x^0) - d(x^0) \le
\epsilon$ to obtain an $\epsilon-$optimal solution of
\eqref{problem}. In particular, if $x^0 = 0$, then the inner
subproblem has the form: $\min\limits_{u \in U} \; f(u) +
\frac{\mu}{2}\text{dist}_{\mathcal{K}}(Gu+g)^2$, which can be seen
as a  differentiable penalty problem. We conclude that, in the case
of known $R_d$, the gradient augmented  Lagrangian method is similar
to the quadratic penalty method. \qed
\end{remark}


\subsubsection{Inexact fast gradient augmented Lagrangian  method}
We further incorporate Nesterov accelerated step into   the
classical augmented Lagrangian method, i.e.  we apply the inexact
fast gradient method on the augmented dual problem.  We analyze the
overall complexity of the inexact fast gradient augmented Lagrangian
method, under the Assumptions \ref{strong_conv} with $\sigma_f = 0$,
\ref{simple_set} and \ref{assump_bound_multipliers}. Using the
inexact oracle relation \eqref{inexact_descentlemma} and Theorem
\ref{ifg_rate_conv} we immediately obtain the following iteration
complexity (in terms of outer iterations) of the fast gradient
method:
\begin{corollary}\label{corr_icfg_fastgrad}
Under  Assumptions \ref{strong_conv} with $\sigma_f=0$,
\ref{simple_set} and \ref{assump_bound_multipliers}, let
$\mu,\delta>0,$ and $(x^k, y^k)_{k \ge 0}$ be the sequences
generated  by Algorithm
\textbf{ICFG}($d^{\text{ag}}_{\mu},0,3\delta, 2L_{\text{d}}$) with
$\theta_{k+1} = \frac{1 + \sqrt{1 + 4 \theta_k^2 }}{2}$ for all $k
\geq 1$.  Then, we have the following estimate on dual
suboptimality:
\begin{equation*}
f^* - d_{\mu}^{\text{ag}}(x^k) \le \frac{4 L_{\text{d}}
R_{\text{d}}^2}{(k+1)^2} + 3 k \delta.
\end{equation*}
\end{corollary}
\noindent Note that the above convergence rate is linked only to the
number of the outer iterations and omits the complexity of solving
the inner subproblem at step $1.$ Before estimating the total
complexity of the process containing the inner and outer levels, we
provide convergence rates for the primal infeasibility and
suboptimality.

\begin{theorem}\label{aug_fastgrad_outer}
Under Assumptions \ref{strong_conv} with $\sigma_f=0$,
\ref{simple_set} and \ref{assump_bound_multipliers}, let $\mu,\delta
> 0$ and $(x^k,y^k)_{k \ge 0}$ be the sequences generated by
Algorithm \textbf{ICFG}$(d^{\text{ag}}_{\mu},0,3\delta,
2L_{\text{d}})$ with $\theta_{k+1} = \frac{1 + \sqrt{1 + 4
\theta_k^2 }}{2}$ for all $k \geq 1$.  Let $u^i = u_{\mu}(x^i)$ be
such that $\mathcal{L}_{\mu}^{\text{ag}}(u^i,x^i) -
d_{\mu}^{\text{ag}}(x^i) \le \delta$ for $0 \leq i \leq k$. Then,
the average primal sequence
$({\hat u}^k)_{k \geq 1}$ defined by \eqref{average_seq} satisfies:\\
$(i)$ The primal infeasibility is bounded sublinearly as follows:
\begin{equation*}
\text{dist}_{\mathcal{K}}\left( G\hat{u}^k +g \right) \le
\frac{8L_{\text{d}} R_{\text{d}}}{k^2} +
8\sqrt{\frac{3L_{\text{d}}\delta}{k}}.
\end{equation*}
$(ii)$ The primal suboptimality gap is bounded by:
\begin{align*} 
- \frac{8L_{\text{d}} R_{\text{d}}^2}{k^2} -
8R_{\text{d}}\sqrt{\frac{3L_{\text{d}}\delta}{k}} \le f(\hat{u}^{k})
- f^* \le  \frac{8 L_\text{d} \norm{x^0}^2}{k^2} + 3 k \delta.
\end{align*}
\end{theorem}

\begin{proof}
We provide the proof of the primal infeasibility and suboptimality
gap bounds in the Appendix A.2.
\end{proof}

\noindent The necessary number of inner iterations that the inexact
fast  gradient augmented  Lagrangian   method has to perform at each
outer iteration is given by \eqref{inner_comp_Lips}. As in the
previous section, in the following result we provide the total
number of projections performed by Algorithm \textbf{ICFG}, for
simple objective functions (i.e. we make the convention that
$L_f=0$) and objective functions with Lipschitz continuous gradient
(i.e. we have  $L_f>0$). Moreover, we assume for simplicity that
$x^0=0$, $R_{\text{d}} > 1$ and $\epsilon$ sufficiently small.

\begin{theorem}
Under Assumptions \ref{strong_conv} with $\sigma_f=0$,
\ref{simple_set} and \ref{assump_bound_multipliers}, let
$\mu,\epsilon,\delta>0$ and $(x^k,y^k)_{k \ge 0}$ be  generated by
Algorithm
\textbf{ICFG}$(d^{\text{ag}}_{\mu},0,3\delta,2L_{\text{d}})$ with
$\theta_{k+1} = \frac{1 + \sqrt{1 + 4 \theta_k^2 }}{2}$ for $k \geq
1$. Assume that at each outer iteration $k$,  Algorithm
\textbf{ICFG}($\frac{\mu}{2}\text{dist}_{\mathcal{K}}(G\cdot + g +
\frac{1}{\mu}x^k)^2,f,0,\mu\norm{G}^2$) (if $f$ is simple and we
make the convention that $L_f=0$) or Algorithm \textbf{ICFG}($
\mathcal{L}^{\text{ag}}_{\mu} (\cdot, x^k), 0, 0,L_f +
\mu\norm{G}^2$)(if $\nabla f$ is $L_f>0$ Lipschitz continuous) is
called to obtain an approximate solution of the inner problem $u^k =
u_\mu(x^k)$ such that $\mathcal{L}^{\text{ag}}_{\mu}(u^k,x^k) -
d^{\text{ag}}_{\mu}(x^k)\le \delta$. Then, by setting the optimal
smoothing parameter:
\begin{equation}\label{parameters_fastgrad_ag_2}
\mu = \frac{16 R_{\text{d}}^2}{\epsilon}  \quad \text{and} \quad
\delta = \frac{\epsilon}{24}
\end{equation}
the average primal point $\hat{u}^k$ defined by \eqref{average_seq}
is $\epsilon-$optimal after a total number of
 \begin{equation*}
k = \left \lceil \frac{14 L_f^{1/2}D_U}{\epsilon^{1/2}} + \frac{56
R_d \norm{G} D_U}{\epsilon} \right \rceil \quad
 \end{equation*}
projections onto the  primal simple feasible set $\mathcal{K} \times
U$.
\end{theorem}

\begin{proof}
First, we observe that if $R_{\text{d}} > 1$, then from  Theorem
\ref{aug_fastgrad_outer}  the number of the outer iterations
$N_{\epsilon}^{\text{out}}$ satisfies:
\begin{equation*}
N_{\epsilon}^{\text{out}} = \left\lceil 4 R_{\text{d}}\left(
\frac{L_{\text{d}}}{\epsilon}\right)^{1/2} \right \rceil =
\left\lceil 4 R_{\text{d}}\left( \frac{1}{\mu\epsilon}\right)^{1/2}
\right\rceil.
\end{equation*}
for any $\mu > 0$, and by forcing both terms in Theorem
\ref{aug_fastgrad_outer} $(ii)$ to have lower magnitudes than
$\epsilon$, then the inner accuracy $\delta$ satisfies:
\begin{align*}
\delta & \overset{\text{Th.}\; \ref{aug_fastgrad_outer} (ii) }{\le}
\min\left\{ \frac{\epsilon}{N_{\epsilon}^{\text{out}}},
\frac{\epsilon^2 N_{\epsilon}^{\text{out}}}{384 R_{\text{d}}^2
L_d}\right\}=\frac{\epsilon^2 N_{\epsilon}^{\text{out}}}{384
R_{\text{d}}^2 L_d}.
\end{align*}
If one chooses $\delta = \frac{\epsilon^2
N_{\epsilon}^{\text{out}}}{384 R_{\text{d}}^2 L_d}$, then this
inequality implies that:
$$ N^{\text{in}}_{\delta} \le 2D_U\sqrt{\frac{2(L_f + \mu\norm{G}^2)}{\delta}} \le
28 D_U \sqrt{\frac{(L_f+ \mu\norm{G}^2)
R_{\text{d}}}{\mu^{1/2}\epsilon^{3/2}}}.$$ For simplicity let $\mu$
satisfy $\mu \ge \frac{L_f}{\norm{G}^2}$. Then, we obtain in this
case the following computational complexity:
\begin{align*}
N^{\text{in}}_{\delta}N_{\epsilon}^{\text{out}} & \le \frac{ 42
\norm{G} D_U R_{\text{d}}^{1/2} \mu^{1/4}}{\epsilon^{3/4}}
 \left[ \frac{4R_{\text{d}}}{(\mu\epsilon)^{1/2}} + 1 \right] \\
 & = \frac{168 \norm{G} D_U R_{\text{d}}^{3/2}}{\mu^{1/4}\epsilon^{5/4}}
 + \frac{ 42\norm{G} D_U R_{\text{d}}^{1/2} \mu^{1/4}}{\epsilon^{3/4}}.
\end{align*}
Minimizing over the set $\{ \mu \in \rset^{} \; | \; \mu \ge
\frac{L_f}{\norm{G}^2} \}$, we obtain that the  best complexity is
attained for $\mu = \max\left\{\frac{L_f}{\norm{G}^2}, \frac{16
R_{\text{d}}^2}{\epsilon} \right\}$. For a sufficiently small
$\epsilon$, the parameter $\mu$ becomes  $\mu = \frac{16
R_{\text{d}}^2}{\epsilon}$, which implies $N_{\epsilon}^{\text{out}}
= 1$ and further leads to: $ \delta = \frac{\epsilon}{24}$. Since
$N_{\epsilon}^{\text{out}} = 1$, under the  above choice  the total
number of projections onto $\mathcal{K} \times U$ required for
attaining an $\epsilon-$optimal point is given by:
\begin{align*}
N_{\epsilon}^{\text{out}} N^{\text{in}}_{\delta} =
N^{\text{in}}_{\delta} & \le \sqrt{\frac{8 D_U^2 (L_f + \mu
\norm{G}^2)}{\delta}} \le \frac{14 L_f^{1/2} D_U}{\epsilon^{1/2}} +
\frac{56 \norm{G} D_U R_{\text{d}}}{\epsilon},
\end{align*}
which proves our result. Note that if we make the convention that
$L_f = 0$, then we get the overall complexity for the case when $f$
is convex and simple function.
\end{proof}

\noindent It can be observed that, for an optimal choice  of the
smoothing parameter $\mu$, the inexact fast gradient augmented
Lagrangian method has the same computational complexity as the
classical inexact gradient augmented Lagrangian method, i.e.
$\mathcal{O}\left( \frac{1}{\epsilon}\right)$ total  projections
onto simple set $\mathcal{K} \times U$. However, we will show the
superiority of the fast variant in Section \ref{sec_comparison},
when we analyze the complexity of first order augmented Lagrangian
methods for attaining the optimality criteria introduced in
\cite{LanMon:15}.


\subsection{Adaptive inexact augmented Lagrangian method}
We have previously seen that, in the optimal case, both classical
and fast augmented Lagrangian methods are dependent on the constant
$\norm{x^*}$ via $R_d$, which in general is unknown a priori.
Therefore, in this section we introduce implementable  variants of
previous first order augmented Lagrangian methods, which approximate
$\norm{x^*}$ at each iteration, but maintain the same optimal
computational complexities with those given in the previous theorems
(up to a logarithmic factor). First, we observe that in the optimal
case (when $\norm{x^*}$ is known), both classical and fast augmented
Lagrangian methods perform a single outer iteration in order to
attain an $\epsilon-$optimal point. Therefore, we can intuitively
apply a search procedure which finds an upper bound on $\norm{x^*}$
in logarithmic number of steps, by performing a single outer
iteration and restarting the augmented Lagrangian method. It is
important to observe that this restarting strategy leads to an
identical scheme for both classical and fast augmented Lagrangian
methods. Throughout this section, we assume that the gradient
$\nabla f$ is Lipschitz continuous with constant $L_f>0$ (when $f$
is simple, with a similar reasoning as given below, we can obtain
the same  computational complexity results).

\vspace{3pt}

\begin{center}
\framebox{
\parbox{10cm}{
\noindent \textbf{ Algorithm A-IAL ($\mu_0, \epsilon$) }

\noindent Initialize $x^0 \in \rset^n$. For $k\geq 0$ do:
\begin{enumerate}
\item Compute $ u^k \in U$ such that $\mathcal{L}_{\mu}^{\text{ag}}(u^k,x^k) - d_{\mu}^{\text{ag}}(x^k) \le \frac{\epsilon}{3}$
\item Update: ${x}^{k+1}= x^k + \mu_k \nabla_x \mathcal{L}_{\mu}^{\text{ag}}(u^k,x^k)$
\item If $\text{dist}_{\mathcal{K}}(Gu^k + g) \le \epsilon$, then \textbf{STOP}
and \textbf{return} $u^k$, otherwise, $k = k+1, \; \mu_{k+1} = 2
\mu_k$ and go to step 1.
\end{enumerate}
}}
\end{center}

\vspace{3pt}

\noindent This adaptive scheme is equivalent with the classical
augmented Lagrangian method but with increasing smoothing parameter.
Further, we present the computational complexity of this algorithm
in the last primal point $u^k$ and compare it with the previous
results.

\begin{theorem}\label{th_a-ial}
Under Assumptions \ref{strong_conv} with $\sigma_f=0$,
\ref{simple_set} and \ref{assump_bound_multipliers}, let $\epsilon,
\mu_0>0$ and $(x^k)_{k \ge 0}$ be the sequence generated by
Algorithm \textbf{A-IAL}$(\mu_0,\epsilon)$. Assume that at each
outer iteration $k \ge 0$, the Algorithm
\textbf{ICFG}($\mathcal{L}^{\text{ag}}_{\mu} (\cdot, x^k), 0, 0, L_f
+ \mu\norm{G}^2$) is called to obtain $u^k$ such that
$\mathcal{L}^{\text{ag}}_{\mu}(u^k,x^k) -
d^{\text{ag}}_{\mu}(x^k)\le \frac{\epsilon}{3}$. Then, after a total
number of:
\begin{equation*}
\left \lceil \log_2 \left( \max\left\{
\frac{16R_d^2}{\mu_0\epsilon},\frac{L_f}{\mu_0\norm{G}^2} \right\}
\right) \right\rceil \left[\left( \frac{6 L_f
D_U^2}{\epsilon}\right)^{1/2} + 1\right] + \frac{16 \norm{G} R_d
D_U}{\epsilon}+ \frac{ 4 L_f^{1/2}D_U}{\epsilon^{1/2}}
\end{equation*}
projections onto the simple set $\mathcal{K} \times U$, the last
primal point $u^k$ satisfies
\begin{equation}\label{our_criteria}
 -\epsilon \norm{x^*} \le f(u^k) - f^* \le \epsilon, \qquad \text{dist}_{\mathcal{K}}(Gu^k+g) \le \epsilon.
\end{equation}
\end{theorem}

\begin{proof}
From Theorem \ref{th_complexity_aug}, it can be seen that the
inexact  gradient augmented  Lagrangian method performs a single
outer iteration if the optimal smoothing parameter $\mu^*
=\max\left\{
\frac{16R_d^2}{\epsilon},\frac{L_f}{\norm{G}^2}\right\}$ is chosen.
Therefore, by iteratively doubling an arbitrary initial value
$\mu_0$ of the smoothing parameter, we attain $\mu^*$ after:
$$N_{\epsilon}^{\text{out}} = \left \lceil \log_2 \left( \frac{\mu^*}{\mu_0} \right) \right\rceil $$
iterations. If the optimal value $\mu^*$ is attained, a single
\textbf{A-IAL} iteration would be sufficient to obtain an
$\epsilon-$optimal primal point. However, since we do not know in
advance $f^*$, we check only the feasibility criterion and stop if a
prespecified accuracy is reached. This stopping criterion ensures
that the  final  point $u^k$, when the algorithm stops, satisfies:
\begin{equation*}
 -\norm{x^*} \epsilon \le f(u^k) - f^* \le \epsilon, \qquad \text{dist}_{\mathcal{K}}(Gu^k+g) \le \epsilon.
\end{equation*}
From \eqref{inner_comp_Lips} it can be seen that the maximal number
of projections performed by the Algorithm \textbf{A-IAL}, in order
to ensure the above set of criteria, is given by:
\begin{align*}
\sum\limits_{k=1}^{N_{\epsilon}^{\text{out}}}
N_{\epsilon,k}^{\text{in}} & =
\sum\limits_{k=1}^{N_{\epsilon}^{\text{out}}}
\left \lceil \sqrt{\frac{6(L_f + \mu_k \norm{G}^2)D_U^2}{\epsilon}} \right\rceil \\
& \le N_{\epsilon}^{\text{out}} + \sum\limits_{k=1}^{N_{\epsilon}^{\text{out}}} \left[ \left(\frac{6L_fD_U^2}{\epsilon}\right)^{1/2} + 2^{k/2} \left(\frac{(6\mu_0)^{1/2} \norm{G}D_U}{\epsilon^{1/2}} \right)\right]  \\
& \le N_{\epsilon}^{\text{out}}\left[\left( \frac{6 L_f D_U^2}{\epsilon}\right)^{1/2} + 1\right] \!+\! 2 \left( \frac{\mu^*}{\mu_0}\right)^{1/2}\! \left( \frac{(3 \mu_0)^{1/2} \norm{G}D_U}{\epsilon^{1/2}}\right)\\
& \le N_{\epsilon}^{\text{out}}\left[\left( \frac{6 L_f D_U^2}{\epsilon}\right)^{1/2} + 1\right] \!+\! \left( \frac{16R_d^2}{\epsilon} \!+\! \frac{L_f}{\norm{G}^2} \right)^{1/2}\! \left( \frac{ 4 \norm{G}D_U}{\epsilon^{1/2}}\right)\\
& \le N_{\epsilon}^{\text{out}}\left[\left( \frac{6 L_f
D_U^2}{\epsilon}\right)^{1/2} + 1\right] + \frac{16 \norm{G} R_d
D_U}{\epsilon}+ \frac{ 4 L_f^{1/2}D_U}{\epsilon^{1/2}},
\end{align*}
which proves our statement.
\end{proof}

\noindent The above result establishes that the Algorithm
\textbf{A-IAL}  performs $\mathcal{O}\left(
\frac{1}{\epsilon}\right)$ total  projections onto simple set
$\mathcal{K} \times U$ in order to obtain a primal point satisfying
\eqref{our_criteria}. Note that the order of the computational
complexity  is the same for the Algorithm \textbf{A-IAL} and  for
the inexact gradient augmented  Lagrangian method. However, this
adaptive scheme \textbf{A-IAL} has the advantage that it is
implementable, i.e. the stopping criterion can be checked and the
parameters of the method are computable.


\subsection{Inexact first order Lagrangian method for a
modified Nesterov smoothing} \label{subsec_nssmoothing} The
smoothing strategy presented in the previous sections is equivalent
with the application of the classical Moreau smoothing technique on
the entire dual function $d$. Unlike this classical approach, we
take in this section a new different path: we make use of the new
separable structure  of the dual function \eqref{dual_separation}
and we only smooth the \textit{Lagrangian} part $d_U$ of the dual
function and keep the nonsmooth  part $d_{\mathcal{K}}$ unchanged.
Based on  \eqref{dual_separation}, we introduce the following smooth
approximation of $d_U$:
\[ d_{U,\mu}(x) = \min\limits_{u \in U} \mathcal{L}_{\mu}(u,x),
\qquad  \text{where} \quad \mathcal{L}_{\mu}(u,x) = f(u) + \langle
x, Gu+g\rangle + \mu \ p_U(u), \]

\vspace{-8pt}

\noindent where $p_U(u)$ is a simple prox-function, continuous and
strongly convex on  $U$. Denote $u_0 =\arg \min_{u \in U} p_U(u)$
and assume without loss of generality that $p_U(u_0) =0$ and its
strong convexity parameter is $1$. Then, we have $p_U(u) \geq 1/2 \|
u - u_0\|^2 \; \forall u \in U$ . One typical example satisfying
these assumptions is $p_U(u) = 1/2 \|u\|^2$.  The function
$d_{U,\mu}$ has Lipschitz continuous gradient \cite{Nes:051}:
\begin{equation*}
\norm{\nabla d_{U,\mu}(x) - \nabla d_{U,\mu}(y)} \le L_{\text{d}}
\norm{x - y} \quad \forall x,y \in \rset^m, \quad \text{with
constant} \quad L_{\text{d}} = \norm{G}^2/\mu.
\end{equation*}

\noindent First, note that if $\mu=0$, then we recover the classical
Lagrangian and dual functions. Secondly, the gradient of $d_{U,\mu}$
satisfies:
$$\nabla d_{U,\mu}(x) = Gu_{\mu}^*(x)+g,  \qquad \text{where} \qquad u_{\mu}^*(x) \in \arg\min\limits_{u \in U} \mathcal{L}_{\mu}(u,x).$$
Moreover, we use in the sequel the following notation for the
partial gradient of $ \mathcal{L}_{\mu}$:
\begin{align*}
\nabla_x \mathcal{L}_{\mu}(u,x) = Gu+g.
\end{align*}
The smoothed dual function $d_{U,\mu}$ leads to a novel smooth
approximation of the composite dual  function $d$, that we aim to
maximize using fast gradient method:
\begin{equation*}
f_\mu^* =  \max\limits_{x \in \rset^m} \; d_{\mu}(x) \quad \left(=
d_{U,\mu}(x) + d_{\mathcal{K}}(x)\right).
\end{equation*}
We denote with $X^*_{\mu} = \arg\max_{x \in \rset^m}  \; d_{\mu}(x)$
the optimal solution set of the smoothed dual problem and
$x^*_{\mu}$ an optimal point. It is important to note that, in many
cases, $u_{\mu}^*(x)$ cannot be computed exactly, but within a
pre-specified accuracy, which leads us to the inexact framework
introduced in the previous section. Thus, in the rest of the section
we define $u_{\mu}(x) \in U$ the inexact solution of the inner
problem satisfying:
\begin{equation}\label{func_approx}
 0 \le  \mathcal{L}_{\mu}(u_{\mu}(x),x) - d_{U,\mu}(x) \le \delta.
\end{equation}
Then, we can  derive a first order  inexact oracle for  the smoothed
dual function $d_{U,\mu}$:

\begin{theorem}
Let $\mu,\delta>0$, then we have the following first order inexact
$(3 \delta, 2L_{\text{d}})$-oracle  for the smoothed dual function
$d_{U,\mu}$:
\begin{align}\label{inexact_oracle}
0 \le \mathcal{L}_{\mu}&(u_{\mu}(y),y) +  \langle \nabla_x
\mathcal{L}_{\mu}(u_{\mu}(y),y), x - y \rangle - d_{U,\mu}(x) \le
\frac{2L_{\text{d}}}{2}\norm{x - y}^2 +  {3}\delta
\end{align}
for all $x,y \in \rset^m$,  where $u_{\mu}(y) \in U$ satisfies
\eqref{func_approx} and $L_d = \frac{\norm{G}^2}{\mu}$.
\end{theorem}

\begin{proof}
The proof is similar with the one  given in Theorem
\ref{th_inexact_oracle} and thus we omit it.
\end{proof}

\noindent The relation \eqref{inexact_oracle} implies that  the
smoothed dual function $d_{U,\mu}$ is equipped with a $(3\delta,
2L_{\text{d}})$-oracle, i.e. $\phi_{\delta,L} (x) =
\mathcal{L}_{\mu}(u_{\mu}(x),x)$ and $\nabla \phi_{\delta,L} (x)=
\nabla_x \mathcal{L}_{\mu}(u_{\mu}(x),x) = Gu_{\mu}(x)+g$. We notice
that there are some previous results on the application of Nesterov
smoothing technique for solving the dual of linear equality
constrained convex problems
\cite{BotHen:15,DevGli:12,NecSuy:08,QuoNec:15,QuoCev:14}, but these
algorithms require exact solution of the inner subproblem and more
conservative convergence estimates are derived. Further, we estimate
the rate of convergence of Algorithm \textbf{ICFG} on the modified
Nesterov smoothing of the dual function. First, let us redefine the
following finite quantity:
\begin{align*} R_{\text{d}} =
\max\limits_{\mu \in \mathcal{C}} \min\limits_{x^*_{\mu} \in
X^*_{\mu}} \; &\norm{x^0 - x^*_{\mu}} < \infty,
\end{align*}
where $\mathcal{C}$ is a compact set in $\rset_+$. From
\cite{NedOzd:09}[Lemma 1] it follows immediately that such an
$R_{\text{d}}$ is always finite for $\mathcal{C} = [0,\ c]$, with $0
<c < \infty$, provided that a Slater vector exists. Note that we can
also bound $\min\limits_{x^* \in X^*} \; \norm{x^0 -x^*} \le
R_{\text{d}}$. Using  Theorem \ref{ifg_rate_conv}, we get the
following estimate on dual suboptimality:

\begin{corollary}\label{corr_icfg_fastgrad_NS}
Under Assumptions \ref{strong_conv} with $\sigma_f=0$,
\ref{simple_set} and \ref{assump_bound_multipliers}, let
$\mu,\delta>0$ and $(x^k, y^k)_{k \ge 0}$ be the sequences generated
by Algorithm \textbf{ICFG}($d_{U,\mu},d_{\mathcal{K}},3\delta,
2L_{\text{d}}$), with $\theta_{k+1} \!\!=\! \frac{1 + \sqrt{1 + 4
\theta_k^2 }}{2}$ for $k \!\!\ge\! 1$. Then,  we have the following
estimate on dual suboptimality:
\begin{equation*}
f_\mu^* - d_{\mu}(x^k) \le \frac{4 L_{\text{d}}
R_{\text{d}}^2}{(k+1)^2} +  3k \delta.
\end{equation*}
\end{corollary}

\noindent We further estimate the rate of convergence of the average
primal sequence generated by \textbf{ICFG} on the modified Nesterov
smoothing of the dual function. For simplicity of the exposition, we
assume further that $ x^0 = 0, \; u^0 = 0, \; R_{\text{d}} \ge 1, \;
\norm{G} > 1, \; D_U> 1$ and $\epsilon < 1$.  However, in the case
when one of these conditions  does not hold, then there is no change
in the order of our results, but slight differences in constants.
Using these simplifications, we obtain the following outer iteration
complexity for \textbf{ICFG}.

\begin{theorem}\label{ns_grad_outer}
Under Assumptions \ref{strong_conv} with $\sigma_f=0$,
\ref{simple_set} and \ref{assump_bound_multipliers}, let
$\mu,\delta>0$ and $(x^k,y^k)_{k \ge 0}$ be the sequence generated
by   Algorithm \textbf{ICFG}($d_{U,\mu},d_{\mathcal{K}},3\delta,
2L_{\text{d}}$) with $\theta_{k+1} = \frac{1 + \sqrt{1 + 4
\theta_k^2 }}{2}$ for all $k \geq 1$.  Define $u^i = u_{\mu}(x^i)$
such that $\mathcal{L}_{\mu}(u^i,x^i) - d_{U,\mu}(x^i) \le \delta $.
For any fixed number of outer iterations $K \ge 1$, if we set
$\mu(K) = \frac{2^{3/2} \|G\| R_{\text{d}}}{D_U K}$, the average
primal point ${\hat
u}^K$ defined by \eqref{average_seq} satisfies:\\
$(i)$ The primal infeasibility is bounded sublinearly as
 follows:
\begin{equation}\label{ns_feasibility_main}
 \text{dist}_{\mathcal{K}}(G\hat{u}^K+g) \le  \frac{2^{3/2} \norm{G}D_U}{K}   +
2\left( \frac{\norm{G} D_U  \delta}{R_d} \right)^{1/2}.
\end{equation}
$(ii)$ The primal suboptimality gap is bounded sublinearly by:
\begin{align}\label{ns_subopt_main}
-\frac{2^{3/2} \norm{G}D_UR_d}{K}  - 2\left( \norm{G} D_U R_d \delta
\right)^{1/2} \leq  f(\hat{u}^K) - f^* \le \frac{2^{3/2}\norm{G}
R_{\text{d}} D_U}{K} + 3 K \delta.
\end{align}
\end{theorem}

\begin{proof}
This proof  is similar to the one given in Appendix A.2. However,
it is also given  in the companion paper  \cite[Appendix
A.3]{NecPat:15}.
\end{proof}

\vspace{5pt}

\noindent It is important to remark that if the functions $f$ and
$p_U$  are simple, then by definition, the solution of the inner
problem $\min_{u \in U} \mathcal{L}_{\mu}(u,x)$ can be found
efficiently (e.g. in linear time or even in closed form). Otherwise,
$\mathcal{L}_{\mu}(\cdot,x)$ is a composition between a function $f$
with $\nabla f$  Lipschitz continuous with constant $L_f>0$, and a
$\mu$-strongly convex and simple function $p_U$  and thus, Nesterov
optimal method for composite problems with a strongly convex part
and a  smooth part
 finds an approximate solution $u_{\mu}(x)$ for the inner problem  satisfying
$\mathcal{L}_{\mu}(u_{\mu}(x),x) - d_{U,\mu}(x) \le \delta$ in
\cite{Nes:13}:
\begin{equation}\label{inner_bound_nesLip}
N_{\delta}^{\text{in}} = \left\lceil \sqrt{\frac{L_f}{\mu}} \log
\left(\frac{ L_f  D_U^2}{4\delta}\right) \right \rceil
\end{equation}
projections onto the simple set $U$. Now, we are ready to derive the
overall  iteration complexity of Algorithm \textbf{ICFG} in this
case:

\begin{theorem}\label{th_modif_ns}
Under Assumptions \ref{strong_conv} with $\sigma_f=0$,
\ref{simple_set} and \ref{assump_bound_multipliers}, let
$\mu,\delta>0$ and $\epsilon>0$, and the sequences $(x^k,y^k)_{k \ge
0}$ be generated by the Algorithm
\textbf{ICFG}($d_{U,\mu},d_{\mathcal{K}},3\delta, 2L_{\text{d}}$)
with $\theta_{k+1} = \frac{1 + \sqrt{1 + 4 \theta_k^2 }}{2}$ for all
$k \geq 1$. Also let $N_{\epsilon}^{\text{out}} = \left\lceil
\frac{6 \norm{G}D_U R_{\text{d}}}{\epsilon} \right\rceil$. Assume
further that  the primal average point $\hat{u}^k$ is given by
\eqref{average_seq}, then the
following assertions hold:\\
$(i)$ If the function $f$ is simple, then by setting an optimal
smoothing parameter:
$$\mu = \frac{2^{3/2} \norm{G} R_d}{D_U N_{\epsilon}^{\text{out}}}  \quad \text{and} \quad \delta=0,$$
the primal average point $\hat{u}^k$ is $\epsilon-$optimal after
$$ k = \left \lceil \frac{6 \norm{G}R_{\text{d}} D_U}{\epsilon} \right \rceil$$
projections onto the primal feasible set $U$ and polar cone $\mathcal{K}^*$.\\
$(ii)$ If the function $f$ is not simple, but $\nabla f$ is
Lipschitz continuous with  $L_f>0$ and, at each outer iteration $k
\ge 1$, Nesterov optimal method for strongly convex and smooth
objective functions \cite{Nes:13} is called to obtain an approximate
optimal  point $u^k = u_\mu (x^k) \in U$ such that
$\mathcal{L}_{\mu} (u^k, x^k) - d_{U,\mu}(x^k) \le \delta$. By
setting an optimal smoothing parameter:
$$\mu = \frac{2^{3/2} \norm{G} R_d}{D_U N_{\epsilon}^{\text{out}}}
 \quad  \text{and} \quad \delta = \frac{\epsilon}{6 N_{\epsilon}^{\text{out}}},$$ then the
average primal point $\hat{u}^k$ is $\epsilon-$optimal after at most
\begin{equation*}
k = \left \lfloor \left(\frac{24 \norm{G}R_{\text{d}} D_U^2
L_f^{1/2}}{\epsilon^{3/2}} \!+\! \frac{12 L_f^{1/2}\norm{G}D_U
R_d}{\epsilon} \right) \left[\log \left(\frac{36
\norm{G}R_{\text{d}} L_f D_U^3}{\epsilon^2}\right)+ 1 \right]
 \right \rfloor
\end{equation*}
projections onto the set $U$ and  $\left \lceil \frac{6
\norm{G}R_{\text{d}} D_U}{\epsilon} \right \rceil$ projections onto
the   cone $\mathcal{K}^*$.
\end{theorem}

\begin{proof}
By forcing both hand sides in \eqref{ns_subopt_main} to be equal
with $\epsilon$, then we obtain:
\begin{align}\label{outer_bound_nes}
 N_{\epsilon}^{\text{out}} & =
 \left\lceil \frac{6 \norm{G}D_U R_{\text{d}}}{\epsilon} \right\rceil
\end{align}
outer projections onto $\mathcal{K}^*$, and the inner accuracy
satisfies (provided that $L_f>0$):
\begin{align}\label{inner_acc_nes}
\delta  \!=\! \min\left\{ \frac{\epsilon^2}{8 \norm{G}D_U
R_{\text{d}}}, \frac{\epsilon}{6 N_{\epsilon}^{\text{out}}} \right\}
 \le \frac{\epsilon^2}{36 \norm{G}D_U R_{\text{d}}}.
\end{align}
Considering the optimal choice of the smoothing parameter (see
\cite[Appendix A.3]{NecPat:15}) and taking into account the bound
\eqref{outer_bound_nes} we get:
$$\mu(N_{\epsilon}^{\text{out}}) = \frac{2^{3/2} \norm{G}R_{\text{d}}}{D_UN_{\epsilon}^{\text{out}}}.$$
Further, using \eqref{inner_acc_nes} and the smoothing parameter
$\mu(N_{\epsilon}^{\text{out}})$ in the inner complexity
\eqref{inner_bound_nesLip}, we get the following bound on the total
number of inner iterations:
\begin{align*}
N_{\epsilon}^{\text{out}} & N^{\text{in}}_{\delta}  \le \frac{12 \norm{G} D_U R_{\text{d}}}{\epsilon} \left( \sqrt{\frac{3L_fD_U^2}{\epsilon}} + \sqrt{L_f} \right) \log \left(\frac{36 \norm{G}R_{\text{d}} L_f D_U^3 }{\epsilon^2}\right) \\
 & \hspace{7cm} + \frac{12 \norm{G}D_U R_{\text{d}}}{\epsilon} \\
& \le \left(\frac{24 \norm{G}R_{\text{d}} D_U^2
L_f^{1/2}}{\epsilon^{3/2}} \!+\! \frac{12 L_f^{1/2}\norm{G}D_U
R_d}{\epsilon} \right) \left[\log \left(\frac{36
\norm{G}R_{\text{d}} L_f D_U^3}{\epsilon^2}\right)+ 1 \right].
\end{align*}
These bounds confirm our result.
\end{proof}

\begin{remark}
If the objective function $f$ is strongly convex, i.e. it satisfies
Assumption \ref{strong_conv} with $\sigma_f>0$, and has Lipschitz
gradient of constant $L_f>0$, then it is well known that the dual
function $d$ has Lipschitz gradient with constant $L_{\text{d}} =
\frac{\norm{G}^2}{\sigma_f}$ \cite{NecPat:14}, and therefore any
smoothing technique is redundant. In this setting, using the first
order inexact oracle framework from previous sections, we can easily
derive overall complexity of Algorithm \textbf{ICFG} for the average
primal point $\hat{u}^k$ of order
$\mathcal{O}\left(\frac{1}{\sqrt{\epsilon}}\log
(\frac{1}{\epsilon^{3/2}})\right)$ projections onto the set $U$ and
$\mathcal{O}(\frac{1}{\sqrt{\epsilon}})$ projections onto the cone
$\mathcal{K}$, see  \cite{NecPat:14} for more details. \qed
\end{remark}

\begin{remark}
 { If we assume that there exists a  bound $R_p$ such
that $\max\limits_{x \in {\cal K}^*} \; \norm{u^0-u(x)} \le R_p <
\infty $, then we can remove the boundedness assumption on $U$ (i.e.
Assumption 2.2 $(ii)$) and all the previous complexity results hold
by replacing  $D_U$ with $R_p$. } \qed
\end{remark}

\noindent In conclusion,  the inexact fast gradient method for  the
modified  Nesterov smoothing of the dual function performs the same
number $\mathcal{O}\left( \frac{1}{\epsilon}\right)$ of projections
onto the cone as the previous inexact first order  augmented
Lagrangian methods. However, in Nesterov smoothing method for smooth
objective functions the number
$\mathcal{O}\left(\frac{1}{\epsilon^{3/2}}\right)$ of projections
onto $U$ is significantly larger  than in the previous augmented
Lagrangian smoothing methods. On the other hand, the optimal
smoothing parameter $\mu$ given in the augmented Lagrangian
framework cannot be fixed a priori due to its dependence on
$\norm{x^*}$ via $R_{\text{d}}$ and thus we need some adaptive
scheme, while the optimal choice of $\mu$ in the case of Nesterov
smoothing strategy can be easily computed in the initialization
phase according to Theorem \ref{th_modif_ns}.


\section{First order penalty methods}
\label{sect_penalty}  {The complexity analysis of
primal-dual methods from Section \ref{sec_Lagrangian} has been based
on the Assumption \ref{assump_bound_multipliers}. Also the most papers on penalty methods make the
strong Assumption \ref{assump_bound_multipliers}, that is there
exists an  optimal Lagrange multiplier for the primal convex problem
\eqref{problem} \cite{LanMon:13}. This property is usually
guaranteed through a Slater type condition, which in the large-scale
settings it is very difficult to check computationally or such a
condition might not even hold. 
In this section we
remove Assumption \ref{assump_bound_multipliers} and  analyze
various penalty strategies for solving the conic constrained convex
optimization problem \eqref{problem} without this assumption.
Therefore, we now consider the conic convex problem \eqref{problem}
which  does not  necessarily admit a Lagrange multiplier that closes
the duality gap. 
To the best of our knowledge this is one of
the first computational complexity results for first order penalty
methods for conic problems when it is not necessarily assumed the
existence of  a Lagrange multiplier that closes the duality gap.}

\noindent  First, denote $f_* = \min\limits_{u \in U} f(u)$. Given
the difficulties induced by the linear conic constraints, the
original problem \eqref{problem} can be reformulated in this case,
using a (non)differentiable penalty function, as an optimization
problem with simple constraints. Therefore, for a penalty parameter
$\rho>0$, the basic penalty reformulations of problem
\eqref{problem} are as follows:
\begin{equation}\label{penalty_diff}
 \min\limits_{u \in U} \quad \psi_{\rho}(u) \quad \left(
 = f(u) + \frac{\rho}{2}\text{dist}_{\mathcal{K}}(Gu+g)^2 \right),
\end{equation}

\vspace{-30pt}

\begin{equation}\label{penalty_nondiff}
 \min\limits_{u \in U} \quad \phi_{\rho}(u) \quad  \left(
 = f(u) + \rho \text{dist}_{\mathcal{K}}(Gu+g) \right).
\end{equation}
Depending on the context, we denote $u^*_{\rho} \in
\arg\min\limits_{u \in U} \psi_{\rho}(u)$ or $u^*_{\rho} \in
\arg\min\limits_{u \in U} \phi_{\rho}(u)$. It is well-known that
both formulations have certain advantages and disadvantages. The
differentiable formulation \eqref{penalty_diff} features good
smoothness properties, but it is regarded as an inexact penalty
problem, i.e. as $\rho \to \infty$ we have $u^*_{\rho} \to u^* \in
U^*$. On the other hand, the nondifferentiable formulation
\eqref{penalty_nondiff} lacks smoothness properties, but in the case
when  optimal Lagrange multipliers for \eqref{problem} exist, there
is a finite threshold $\rho^* >0$ such that for any $\rho \ge
\rho^*$, we have $u^*_{\rho} = u^* \in U^*$. We recall the convexity
property of the distance:
\begin{equation}\label{conv_distcone}
 \text{dist}_{\mathcal{K}}(Gu+g) \ge \text{dist}_{\mathcal{K}}(Gv+g) +
 \langle G^T s(v), u-v  \rangle
 \quad \forall u,v \in \rset^m,
\end{equation}
where $ s(v)  \in \partial \text{dist}_{\mathcal{K}}(Gv+g)$ denotes
a subgradient at $v$ of function $\text{dist}_{\mathcal{K}}(G\cdot
+g)$. From \eqref{conv_distcone}, it can be easily seen that for any
$u \in \rset^m$ such that $Gu+g \in \mathcal{K}$ results:
\begin{equation}\label{conv_distcone_case}
\langle s(v), Gv- Gu   \rangle \ge \text{dist}_{\mathcal{K}}(Gv+g)
 \quad \forall v \in \rset^m.
\end{equation}
Further we analyze both penalty strategies combined with fast
gradient method and we derive the overall  complexities for them.


\subsection{Fast gradient differentiable penalty method}
If the gradient $\nabla f$ is $L_f>0$ Lipschitz continuous, then the
penalty function $\psi_{\rho}$ has also Lipschitz continuous
gradients with constant $L_{\psi} = L_f + \rho \norm{G}^2$. Note
that the optimality conditions of \eqref{penalty_diff} are:
\begin{equation}\label{optim_cond_diff}
 \langle \nabla f (u^*_{\rho}) + \rho \text{dist}_{\mathcal{K}}(Gu^*_{\rho}+g) G^T s(u^*_{\rho}), u - u^*_{\rho} \rangle \ge 0 \quad \forall u \in U.
\end{equation}
Now, we state our result regarding the computational complexity of
the penalty method with differentiable penalty, regarding simple
objective functions (set $L_f=0$ in the complexity estimate) or
smooth objective functions with Lipschitz continous gradients (i.e.
$L_f>0$). Define $\Delta^* = f^* - f_*$.

\begin{theorem} \label{th_smooth_diff}
Under Assumptions \ref{strong_conv} with $\sigma_f = 0$ and
\ref{simple_set}, let $\rho>0, \epsilon \in (0,\Delta^*/2)$ and
$(u^k,v^k)_{k \ge 0}$ be the sequence generated by the Algorithm
\textbf{ICFG}$(\psi_{\rho},0,0,L_{\psi})$ with $\theta_{k+1} =
\frac{1 + \sqrt{1 + 4 \theta_k^2}}{2}$ for $k \ge 1$. If the penalty
parameter satisfies:
\begin{equation}
\label{penalty_bound}
 \rho  \ge \frac{4\Delta^*}{\epsilon^2}
\end{equation}
and $f$ is simple ($L_f = 0$) or $\nabla f$ Lipschitz continuous
($L_f > 0$), then after
$$k=  \left\lceil \sqrt{\frac{2L_fD_U^2 }{\epsilon}} +
 \frac{(8\Delta^*)^{1/2}\norm{G} D_U}{\epsilon^{3/2}}\right\rceil
 $$
projections onto the simple set $\mathcal{K} \times U$, we have:
\begin{equation}
\label{optimalitycond} -\Delta^* \le f(u^k) - f^* \le \epsilon,
\qquad \text{dist}_{\mathcal{K}}(Gu^k + g) \le \epsilon.
\end{equation}
\end{theorem}

\begin{proof}
First, observe that by taking $u = u^*$ in \eqref{optim_cond_diff}
and using \eqref{conv_distcone_case}, we obtain:
\begin{equation*}
 \langle \nabla f (u^*_{\rho}), u^* - u^*_{\rho}\rangle \ge \rho \text{dist}_{\mathcal{K}}(Gu^*_{\rho}+g)^2.
\end{equation*}
Taking into account that $f(u^*_{\rho}) \ge f_*$, then from the
convexity property of $f$ results:
\begin{equation*}
 \text{dist}_{\mathcal{K}}(Gu^*_{\rho}+g) \le \sqrt{\frac{f(u^*) - f_*}{\rho}} = \sqrt{\frac{\Delta^*}{\rho}}.
\end{equation*}
Therefore, a sufficient condition for
$\text{dist}_{\mathcal{K}}(Gu^*_{\rho} +g) \le \epsilon/2$ is $\rho
\ge \frac{4\Delta^*}{\epsilon^2}$.
Let $\bar{u} \in U$ satisfying:
\begin{equation}\label{eps_subopt}
 f(\bar{u}) + \frac{\rho}{2}\text{dist}_{\mathcal{K}}(G\bar{u}+g)^2 - f(u^*_{\rho}) -
 \frac{\rho}{2}\text{dist}_{\mathcal{K}}(Gu^*_{\rho}+g)^2= \psi_{\rho}(\bar{u}) - \psi_{\rho}^* \le \epsilon.
\end{equation}
Using the convexity property of $f$, \eqref{conv_distcone_case} and
\eqref{optim_cond_diff}, then the relation \eqref{eps_subopt}
implies:
\begin{align*}
& \epsilon \ge \frac{\rho}{2}\text{dist}_{\mathcal{K}}(G\bar{u}
+g)^2  -
 \frac{\rho}{2}\text{dist}_{\mathcal{K}}(Gu^*_{\rho} + g)^2  + \langle \nabla f(u^*_{\rho}),
\bar{u} - u^*_{\rho}\rangle \nonumber\\
 & \overset{\eqref{optim_cond_diff}}{\ge} \frac{\rho}{2}\text{dist}_{\mathcal{K}}(G\bar{u} + g)^2  -  \frac{\rho}{2}\text{dist}_{\mathcal{K}}(Gu^*_{\rho} + g)^2  + \rho \text{dist}_{\mathcal{K}}(Gu^*_{\rho}+g) \langle G^T
 s(u^*_{\rho}),  u^*_{\rho} - \bar{u}\rangle \nonumber\\
 & \overset{\eqref{conv_distcone_case}}{\ge} \frac{\rho}{2}\text{dist}_{\mathcal{K}}(G\bar{u} + g)^2  +
 \frac{\rho}{2}\text{dist}_{\mathcal{K}}(Gu^*_{\rho} + g)^2 -
 \nonumber\\
& \hspace{20pt}  - \rho \text{dist}_{\mathcal{K}}(Gu^*_{\rho}+g)
\left( \text{dist}_{\mathcal{K}}(Gu^*_{\rho}+g) +  \langle
G^T s(u^*_{\rho}), \bar{u} - u^*_{\rho} \rangle \right) \nonumber\\
 &=  \frac{\rho}{2}\left[ \text{dist}_{\mathcal{K}}(Gu^*_{\rho}+g) - \text{dist}_{\mathcal{K}}(G\bar{u}+g)\right]^2.
\end{align*}
The last relation  leads to:
\begin{equation*}
\text{dist}_{\mathcal{K}} ( G\bar{u} +g) \le
\sqrt{\frac{2\epsilon}{\rho}} + \text{dist}_{\mathcal{K}} ( G
u^*_{\rho} +g).
\end{equation*}
For a penalty parameter satisfying \eqref{penalty_bound} and
$\epsilon \le \Delta^*/2$, we reach $\epsilon-$infeasibility:
\begin{equation*}
 \text{dist}_{\mathcal{K}}(G\bar{u}+g) \le \frac{\epsilon^{3/2}}{\sqrt{2\Delta^*}} + \frac{\epsilon}{2} \le \epsilon.
\end{equation*}
To obtain suboptimality bounds, first note that the left inequality
stating $f(\bar{u}) - f(u^*) \ge -\Delta^*$ is trivial. Second, the
relation \eqref{eps_subopt} implies:
\begin{equation*}
  f(\bar{u}) - f(u^*) \le f(\bar{u}) + \frac{\rho}{2}\text{dist}_{\mathcal{K}}(G\bar{u} + g)^2
  - f(u^*)  \le  \psi_{\rho}(\bar{u}) - \psi_{\rho}^* \le \epsilon.
\end{equation*}
By choosing $\rho \ge \frac{4\Delta^*}{\epsilon^2}$ and solving the
differentiable penalty problem \eqref{penalty_diff} with accuracy
$\epsilon$ leads to an $\epsilon-$optimal point of the original
problem \eqref{problem} which satisfies optimality criteria
\eqref{optimalitycond}. For any fixed penalty parameter $\rho > 0$,
Algorithm \textbf{ICFG}$(\psi_{\rho},0,0,L_{\psi})$ generates a
sequence $(u^k)_{k \ge 0}$ with the convergence rate (see Theorem
\ref{ifg_rate_conv}):
\begin{equation*}
  \psi_{\rho}(u^k) - \psi_{\rho}^* \le \frac{2(L_f + \rho\norm{G}^2)D_U^2}{(k+1)^2}.
\end{equation*}
This rate of convergence implies that after
\begin{equation*}
k = \left\lceil \sqrt{\frac{2(L_f +
\rho\norm{G}^2)D_U^2}{\epsilon}}\right\rceil
\end{equation*}
projections onto $\mathcal{K} \times U$, we get $\psi_{\rho}(u^k)-
\psi^*_{\rho} \le \epsilon$. Further, taking into account the
estimation of the penalty  parameter \eqref{penalty_bound}, we can
bound the previous estimate as:
\begin{equation*}
\left\lceil \sqrt{\frac{2L_fD_U^2 }{\epsilon}} +
\sqrt{\frac{8\Delta^*\norm{G}^2 D_U^2}{\epsilon^3}}\right\rceil.
\end{equation*}
Note that the last estimate implies our result.
\end{proof}

\noindent The following simple example shows the tightness of our
result given in Theorem \ref{th_smooth_diff}.

\begin{example}\label{example}
\noindent Given $p > 1$, consider the following convex problem:
\begin{align*}
 \min_{u \in \rset^2}  \quad f(u) \quad \left( := u_2 \right) \quad \text{s.t.}  \quad |u_2|^p \le u_1, \quad u_1=0,
\end{align*}
where $U = \{u \in \rset^2 | \;  |u_2|^p \le u_1 \}$. Note that the
feasible set contains only the trivial point $(0,0)$, and implicitly
we have $u_1 \ge 0$. The Slater condition does not hold in this
case. First, we show that this optimization problem does not admit a
Lagrange multiplier  closing the duality gap. The dual problem of
the above example is given by:
\begin{align*}
  \max\limits_{x \in \rset^{}} \min_{u \in \rset^2}  \quad u_2  + x u_1 \quad \text{s.t.}  \quad |u_2|^p \le u_1.
\end{align*}
Since the objective function is linear,  an equivalent form of the
dual problem is:
\begin{align*}
 \max\limits_{x \in \rset^{}} \min_{u} & \;\; \pm u_1^{1/p} + x u_1.
\end{align*}
Considering the case $u_2 = - u_1^{1/p}$ (for the other case we can
use the same reasoning),  {with the implicit
constraint $u_1 \ge 0$,} the optimal solution $u^*_1$ of this
minimization subproblem is given by: $u^*_1 = (p
x)^{\frac{p}{1-p}}.$ Replacing this value into the cost, and taking
into account that we have to keep $u_1^* \ge 0$, then we obtain the
dual problem:
\begin{equation*}
  {\sup_{x \ge 0}} \; \;\left (\frac{1}{px} \right)^{\frac{1}{p-1}}\left( \frac{1}{p} - 1 \right).
\end{equation*}
The dual function is negative for any $x \ge 0$, and thus we do not
have a bounded  Lagrange multiplier attaining the supremum. Further
we estimate the value of the penalty parameter $\rho$ such that we
get $\epsilon-$infeasibility for $u_{\rho}^*$. The quadratic penalty
reformulation is given~by:
\begin{align}
 \min_{u \in \rset^2}  \quad u_2 + \frac{\rho}{2}u_1^2 \label{examp_penalty} \quad   \text{s.t.}  \quad |u_2|^p \le u_1. \nonumber
\end{align}
Observe that the minimizer $u^*_\rho$ of the above problem is on the
boundary of the feasible set, i.e. $|u_2|^p = u_1$. Then, we get the
following equivalent problem:
\begin{align*}
 \min_{u_2 \in \rset^{}} & \quad u_2 + \frac{\rho}{2}u_2^{2p}.
\end{align*}
The optimality condition of the above problem is given by $ 1 + \rho
p [(u^*_\rho)_2]^{2p-1} = 0,$ which immediately implies:
\begin{equation}
(u^*_{\rho})_2 = \left( - \frac{1}{p \rho}\right)^{\frac{1}{2p -1}}.
\end{equation}
From this expression and the fact that $|(u^*_\rho)_2|^p =
(u^*_{\rho})_1$, it can be derived that $\epsilon-$infeasibility is
attained, i.e. $|(u^*_{\rho})_1| \le \epsilon$, provided that the
penalty parameter satisfies:
\begin{equation*}
   \rho \ge \frac{1}{p} \left( \frac{1}{\epsilon}\right)^{2- \frac{1}{p}} = \frac{\epsilon^{1/p}}{p} \left( \frac{1}{\epsilon}\right)^{2}.
\end{equation*}
Observing that $\frac{\epsilon^{1/p}}{p}$ is a convex function of
$p$, the minimal value of this expression is attained for $p^* =
\ln(1/\epsilon)$. Replacing this value in the above estimate, we
have:
\begin{equation*}
   \rho \ge \frac{\epsilon^{\frac{1}{\ln(1/\epsilon)}}}{\ln(1/\epsilon)} \left( \frac{1}{\epsilon}\right)^{2}
   = \frac{1}{e \ln(1/\epsilon)} \left( \frac{1}{\epsilon}\right)^{2},
\end{equation*}
where $e$ is the Euler constant. Therefore, for this example, the
penalty parameter should satisfy $\rho = \mathcal{O}\left(
\frac{1}{\epsilon^2}\right)$ (up to a logarithmic factor), which
confirms the tightness of our result given in Theorem
\ref{th_smooth_diff}. \qed
\end{example}


\subsection{Fast gradient nondifferentiable penalty method}
\label{subsec_nondif} Given the nonsmoothness feature of the penalty
function $\phi_{\rho}$, we replace  the nonsmooth term
$\text{dist}_{\mathcal{K}}(G \cdot + g)$ with a basic smooth
approximation. Thus, for a given smoothing parameter $\mu>0$, we
replace the original problem with the following smooth  problem:
\begin{equation}\label{smooth_nondif_pen}
\min\limits_{u \in U} \; \; \phi_{\rho,\mu}(u) \;\quad \left(= f(u)+
\rho\sqrt{\text{dist}_{\mathcal{K}}(Gu+g)^2 + \mu^2}\right).
\end{equation}
Note that if $\nabla f$ is Lipschitz continuous with constant
$L_f>0$, then $\nabla \phi_{\rho,\mu}$ is Lipschitz continuous with
constant  {$L_{\phi} = L_f + \frac{\rho
\norm{G}}{\mu}$}. We denote $u^*_{\mu} \in \arg\min\limits_{u \in U}
\; \phi_{\rho,\mu}(u)$ and, for simplicity, assume that  $\Delta^*
\ge \epsilon$ (otherwise some minor changes in constants will
occur).

\begin{theorem}\label{th_smooth_nondiff}
Under Assumptions \ref{strong_conv} with $\sigma_f=0$ and
\ref{simple_set}, let $\mu,\rho, \epsilon
>0 $ and the sequence $(u^k,v^k)_{k \ge 0}$ be generated by the
Algorithm \textbf{ICFG}$(\phi_{\rho,\mu},0,0,L_{\phi})$ with
$\theta_{k+1} = \frac{1 + \sqrt{1 + 4\theta_k^2}}{2}$ for all $k \ge
1$. If the following conditions hold:
\begin{equation}\label{penalty_bound_nondif}
 \rho  = \frac{2\Delta^*}{\epsilon} + 1 \quad \text{and} \quad \mu = \frac{\epsilon}{2},
\end{equation}
and $f$ is simple (convention $L_f = 0$) or $\nabla f$ Lipschitz
continuous ($L_f > 0$), then after
$$
k=  \left\lceil \sqrt{\frac{2L_f D_U^2}{\epsilon}} + \sqrt{\frac{12
\Delta^* \norm{G}D_U^2}{\epsilon^3}} \right\rceil
 $$
projections onto the primal simple feasible set $\mathcal{K} \times
U$, we have:
\begin{equation*}
-\Delta^* \le f(u^k) - f^* \le \epsilon, \qquad
\text{dist}_{\mathcal{K}}(Gu^k + g) \le \epsilon.
\end{equation*}
\end{theorem}
\begin{proof}
\noindent Let $\bar{u} \in U$ be an $\epsilon-$optimal point for the
smoothed penalty problem \eqref{smooth_nondif_pen} satisfying
$\phi_{\rho,\mu}(\bar{u}) - \phi_{\rho,\mu}^* \le \epsilon$, i.e. we
have:
\begin{align}\label{eps_subopt_nondif}
 f(\bar{u}) + \rho \sqrt{\text{dist}_{\mathcal{K}}(G\bar{u}+g)^2 + \mu^2} - f(u^*_{\mu}) &- \rho \sqrt{\text{dist}_{\mathcal{K}}(Gu^*_{\mu}+g)^2 + \mu^2} \le \epsilon.
\end{align}
First, the relation \eqref{eps_subopt_nondif} implies the following:
\begin{align}\label{subopt_right_nondiff}
f(\bar{u}) &-f^* \le f(\bar{u}) + \rho \sqrt{\text{dist}_{\mathcal{K}}(G\bar{u}+g)^2 + \mu^2} - f^* - \rho \mu \nonumber\\
& \le f(\bar{u}) + \rho
\sqrt{\text{dist}_{\mathcal{K}}(G\bar{u}+g)^2 + \mu^2} -
f(u^*_{\mu}) - \rho \sqrt{\text{dist}_{\mathcal{K}}(Gu^*_{\mu}+g)^2
+ \mu^2} \le \epsilon.
\end{align}
Second, from \eqref{eps_subopt_nondif} we have the following
feasiblity relation:
\begin{align*}
 \text{dist}_{\mathcal{K}}(G\bar{u} + g)
&\le \sqrt{\text{dist}_{\mathcal{K}}(G\bar{u} + g)^2 + \mu^2} \\
& \overset{\eqref{subopt_right_nondiff}}{\le} \frac{f(u^*_{\mu}) + \rho \sqrt{\text{dist}_{\mathcal{K}}(Gu^*_{\mu} + g)^2 + \mu^2} - f(\bar{u}) + \epsilon}{\rho} \\
& \le \frac{f^* - f_* + \epsilon}{\rho} + \mu = \frac{\Delta^* +
\epsilon}{\rho} + \mu.
 \end{align*}
Therefore, choosing the parameters conformal to
\eqref{penalty_bound_nondif}, any point satisfying
\eqref{eps_subopt_nondif} is $\epsilon-$optimal in the optimality
criteria \eqref{optimalitycond}.
\noindent Given arbitrary $\mu,\rho>0$, the Algorithm
\textbf{ICFG}($\phi_{\rho,\mu},0,0,L_{\phi} $) applied on the
smoothed  problem \eqref{smooth_nondif_pen} generates
 primal sequences $(u^k,v^k)_{k \ge 0}$ satisfying the following convergence rate (see Theorem \ref{ifg_rate_conv}):
\begin{equation*}
 \phi_{\rho, \mu}(u^k) - \phi_{\rho,\mu}^* \le \frac{2\left(
 L_f + \rho \frac{\norm{G}}{\mu}\right)D_U^2}{k^2}.
\end{equation*}
Thus, the $\epsilon-$suboptimality for problem
\eqref{smooth_nondif_pen} is attained after at most:
\begin{equation*}
 \left\lceil \sqrt{\frac{2L_f D_U^2}{\epsilon}} +
 \sqrt{\frac{2\rho \norm{G} D_U^2}{\mu\epsilon}} \right\rceil
\end{equation*}
projections onto the set $\mathcal{K} \times U$. Taking into account
the assumptions \eqref{penalty_bound_nondif}, we obtain the
computational complexity estimate given in the theorem.
\end{proof}

\vspace{2pt}

\begin{remark}
It is easy to prove  that if the objective function $f$ is strongly
convex, i.e. it satisfies Assumption \ref{strong_conv} with
$\sigma_f>0$, then the differentiable and nondifferentiable penalty
methods from previous sections have computational complexity in the
last primal point $u^k$ of order
$\mathcal{O}\left(\frac{1}{\epsilon}\log
(\frac{1}{\epsilon})\right)$ projections onto the set $\mathcal{K}
\times U$. \qed
\end{remark}

\vspace{2pt}

\noindent From previous discussion it follows that the optimal
penalty parameter $\rho$  depends on $\Delta^*$, which in general is
unknown a priori. Therefore, in the next section we introduce
implementable variants of previous first order penalty methods,
which approximate $\Delta^*$ at each iteration, but maintain the
same optimal computational complexities with those given in the
previous theorems (up to a logarithmic factor).

\subsection{Adaptive fast gradient penalty method}
In this section, regardless of the type of penalty function, we
introduce an Adaptive Penalty Method (A-PM), which rely on a
sequential increase of the penalty parameter $\rho$ until a
satisfactory value is attained.

\vspace{5pt}

\begin{center}
\framebox{
\parbox{12.5cm}{
\noindent \textbf{ Algorithm A-PM ($\rho_0, \epsilon, s $) }
\begin{enumerate}
\item[1.] Set $k=0$ and choose $u_0 \in U$. If $s=``N''$
 choose $\mu>0$. For $k \ge 0$ do:
\item[2.] Apply the Algorithm \textbf{ICFG} on the (smoothed) penalty subproblem and find $u^k$ such that:
\begin{align*}
\psi_{\rho_k}(u^k) - \psi_{\rho_k}^* \le \epsilon, & \;\; \text{if} \;\;\; s=``D''; \\
\phi_{\rho_k}(u^k) - \phi_{\rho_k}^* \le \epsilon, &  \;\;
\text{if}\;\;\; s=``N''.
\end{align*}
\item[3.] If the iterate $u^k$ satisfies $\text{dist}_{\mathcal{K}}(Gu^k+ g) \le \epsilon,$
then \textbf{STOP}. Otherwise, set $\rho_{k+1}=2\rho_k, k = k+1$ and
go to step 2.
\end{enumerate}
}}
\end{center}

\noindent In the previous sections we have seen that, in the general
case, when the optimal Lagrange multipliers do not necessarily
exist, there is a penalty parameter $\bar{\rho}$ dependent on the
type of penalty function, i.e.: $ \bar{\rho} = \begin{cases}
                 \frac{4\Delta^*}{\epsilon^2}, &\text{for smooth penalty} \\
                  \frac{3\Delta^*}{\epsilon}, &\text{for nonsmooth penalty}
                \end{cases},
$ such that if $\rho_k \ge \bar{\rho}$ and $\epsilon \le
\Delta^*/2$, then $u^k$ satisfies \eqref{optimalitycond} and the
algorithm stops. Further, we provide the  computational complexity
for Algorithm \textbf{A-PM} in the case when $\nabla f$ is Lipschitz
continuous with constant $L_f>0$. The complexity results for the
case when $f$ is simple can be derived  similarly.

\begin{theorem}
Under the assumptions of Theorem \ref{th_smooth_diff}, let
$\rho_0,\epsilon> 0$ and the sequence $(u^k)_{k \ge 0}$ be generated
by Algorithm \textbf{A-PM}$(\rho_0,\epsilon, s)$. For
nondifferentiable penalty case assume $\mu= \frac{\epsilon}{2}$.
After a total number of projections onto $\mathcal{K} \times U$
given by:
\begin{equation*}
\begin{cases}
\left\lceil N_{\epsilon}^{\text{out}} \left(\frac{4L_fD_U^2}{\epsilon}\right)^{1/2} +  \frac{24 (\rho_0\Delta^*)^{1/2}\norm{G} D_U }{\epsilon^{3/2}}\right\rceil ,   & \text{for smooth penalty} \\
\left \lceil N_{\epsilon}^{\text{out}}
\left(\frac{4L_fD_U^2}{\epsilon}\right)^{1/2} +
\frac{30(\rho_0\Delta^*\norm{G})^{1/2}D_U}{\epsilon^{3/2}}
\right\rceil, & \text{for nonsmooth penalty},
\end{cases}
\end{equation*}
where $N_{\epsilon}^{\text{out}}  = \begin{cases}
   \left\lceil \log \left( \frac{4\Delta^*}{\epsilon^2 \rho_0}\right) \right\rceil, &\text{for smooth penalty} \\
   \left\lceil \log \left( \frac{3\Delta^*}{\epsilon \rho_0}\right) \right \rceil, &\text{for nonsmooth penalty}
  \end{cases}
 $, the primal point $u^k$ satisfies primal suboptimality $f(u^k) - f^* \le \epsilon$  and primal infeasibility $\text{dist}_{\mathcal{K}}(Gu^k + g) \le \epsilon$.
\end{theorem}

\begin{proof}
The proof follows similar lines as in Theorem \ref{th_a-ial}. It can
be easily seen that, independently of the assumptions on the
objective function $f$,  Algorithm \textbf{A-PM} requires:
\begin{align*}
N_{\epsilon}^{\text{out}} & = \begin{cases}
   \left\lceil \log \left( \frac{4\Delta^*}{\epsilon^2 \rho_0}\right) \right\rceil, &\text{for smooth penalty} \\
   \left\lceil \log \left( \frac{3\Delta^*}{\epsilon \rho_0}\right) \right \rceil, &\text{for nonsmooth penalty}
  \end{cases}
\end{align*}
outer steps to attain an $\epsilon-$optimal point.
Taking into account that in the nonsmooth case, we apply the
classical smoothing strategy from Section \ref{subsec_nondif}, the
iteration complexity for solving the inner subproblem, at outer
iteration $k$, can be bounded by:
\begin{equation*}
N_{\epsilon,k}^{\text{in}} =
\begin{cases}
   \left\lceil \left(\frac{2L_fD_U^2}{\epsilon}\right)^{1/2} + \rho_k^{1/2} \left(\frac{2\norm{G}D_U}{\epsilon^{1/2}}\right)\right\rceil, \;\; &\text{for smooth penalty} \\
   \left\lceil \left( \frac{2L_f D_U^2}{\epsilon} \right)^{1/2} + \rho_k^{1/2} \left[\frac{2\norm{G}D_U}{(\mu \epsilon)^{1/2}} \right] \right\rceil, \;\; &\text{for nonsmooth penalty},
\end{cases}
\end{equation*}
where $\mu>0$ is the smoothing parameter. Knowing the maximal number
of outer stages, note that the total number of the fast gradient
iterations can be computed by summation of all quantities
$N_{\epsilon,k}^{\text{in}}$. Observing that
$\sum\limits_{k=0}^{N_{\epsilon}^{\text{out}}} \rho_k^{\frac{1}{2}}
\le 6\rho_0 2^{\frac{N_{\epsilon}^{\text{out}}}{2}}$, then we obtain
the following bound on the overall complexity:
\begin{equation*}
\sum\limits_{k=0}^{N_{\epsilon}^{\text{out}}}  N_{\epsilon,k}^{\text{in}} \leq
\begin{cases}
N_{\epsilon}^{\text{out}} \left(\frac{2L_fD_U^2}{\epsilon}\right)^{1/2} +  \frac{24 (\rho_0\Delta^*)^{1/2}\norm{G} D_U }{\epsilon^{3/2}} + 1,   & \text{for smooth penalty} \\
N_{\epsilon}^{\text{out}}
\left(\frac{2L_fD_U^2}{\epsilon}\right)^{1/2} +
\frac{30(\rho_0\Delta^*)^{1/2}\norm{G}D_U}{\epsilon^{3/2}} + 1, &
\text{for nonsmooth penalty},
\end{cases}
\end{equation*}
which proves the statements of the theorem.
\end{proof}

\begin{remark}
 { If we assume that there exist a bound  $R_p$ such
that $\norm{u^0 - u^*_{\rho}} \le R_p < \infty $ for all $u^*_{\rho}
\in \arg\min\limits_{u \in U} \; \psi_{\rho}(u)$ \;(or \
$\phi_{\rho}(u))$, then we can remove the boundedness assumption on
$U$ (i.e. Assumption 2.2 $(ii)$) and all the previous complexity
results hold by replacing  $D_U$ with $R_p$. } \qed
\end{remark}

\noindent In conclusion, if we do not assume the existence of an
optimal Lagrange multiplier that closes the duality gap  for the
cone constrained convex problem \eqref{problem}, the  computational
complexity of fast gradient penalty methods, in the worst-case, is
of order $\mathcal{O}(\frac{1}{\epsilon^{3/2}})$.  Moreover, these
bound are tight as Example \ref{example} shows.



\section{Comparisons with  previous work}
\label{sec_comparison} We now present a brief comparison of our
computational complexity results  on  Lagrangian and penalty methods
with previous complexity results from the literature  in various
optimality criteria.   We start comparing the computational
complexity results on  (fast) gradient augmented Lagrangian methods
in the optimality criteria used in this paper: $|f(u_{\epsilon})
-f^* | \le  \epsilon$  and $
\text{dist}_{\mathcal{K}}(Gu_{\epsilon}+g) \le \epsilon$. Various
first order augmented Lagrangian methods have been developed in e.g.
\cite{AybIye:13, LanMon:15} and computational complexity estimates
of order $\mathcal{O}\left( \frac{1}{\epsilon} \right)$ have been
obtained in our criteria. For example,  in \cite{AybIye:13}, an
adaptive augmented Lagrangian method for cone constrained convex
optimization models was analyzed. The authors in  \cite{AybIye:13}
prove that the outer complexity is of order
$\mathcal{O}(\log(1/\epsilon))$ and the inner accuracy is
constrained to be of order $\delta_k = \mathcal{O}\left(\frac{1}{k^2
\beta^k} \right)$, where $\beta>1$, and thus the overall complexity
is similar to the estimates given  in our paper (up to a logarithmic
factor). However, our augmented Lagrangian algorithms can be easily
implemented in practice, their parameters are easy to compute  and
our analysis based on the inexact oracle framework is more simple
and intuitive in comparison with those given in \cite{AybIye:13,
LanMon:15}, opening various possibilities for extensions to more
complex optimization models.

\vspace{3pt}

\noindent On the other hand,  Lan et. al. in \cite{LanMon:15}
considered the linear equality constrained case (i.e. $\mathcal{K} =
\{0\}$) and used another set of $\epsilon-$optimality criteria, i.e.
any $u_{\epsilon} \in U$ is $\epsilon-$optimal if there exists
$x_{\epsilon} \in \rset^m$ satisfying:
\begin{equation}
\label{criteria_Lan}
  \nabla f(u_{\epsilon}) + G^Tx_{\epsilon} \in - \mathcal{N}_U(u_{\epsilon}) + \mathcal{B}_{\epsilon}(0) \quad \text{and} \quad   \norm{G u_{\epsilon} + g} \le \epsilon.
\end{equation}
In these criteria, without any regularization of the original
problem,  the gradient augmented Lagrangian algorithm \textbf{I-AL}
introduced in \cite{LanMon:15} has computational complexity of order
$\mathcal{O}\left( \epsilon^{-\frac{7}{4}}\right)$. We further show
that, using our approach we obtain a suboptimal point satisfying
\eqref{criteria_Lan}, with a much better iteration complexity for
the same algorithm \textbf{I-AL}. More precisely, with our analysis,
Theorem \ref{aug_grad_outer} leads to the fact that \textbf{I-AL}
method from \cite{LanMon:15} should perform: $
N_{\epsilon,1}^{\text{out}} = \left \lceil \frac{16 R_d}{3\mu
\epsilon} \right \rceil $ outer iterations with inner accuracy
$\delta = \frac{\mu \epsilon^2}{128}$. Therefore, denoting the inner
complexity $N_{\delta}^{\text{in}} \le  \sqrt{\frac{2^9 (L_f + \mu
\norm{G}^2)D_U^2}{\mu \epsilon^2}}$, the first stage of
\textbf{I-AL} method of \cite{LanMon:15} requires $
N^{\text{out}}_{\epsilon,1} N_{\delta}^{\text{in}}$ projections
onto $U$ and, on the other hand, the \textbf{Postprocessing}
procedure in \textbf{I-AL} of \cite{LanMon:15} performs $
\left\lceil \frac{2^{5/2} (L_f + \mu\norm{G}^2)D_U}{\epsilon} \right\rceil$ projections onto
$U$. Using these bounds, for any $\mu \ge \frac{L_f}{\norm{G}^2}$,
the total number of projections required by the \textbf{I-AL} method
in \cite{LanMon:15} is bounded with our analysis  by: $\frac{2^{10}
\norm{G} D_U R_d}{3\mu \epsilon^2} + \frac{2^{7/2} \mu \norm{G}^2
D_U}{\epsilon}$. For an optimal complexity, we choose the smoothing
parameter as $ \mu = \frac{2^{13/4} R_d^{1/2}}{ \norm{G}^{1/2}
\epsilon^{1/2}} + \frac{L_f}{\norm{G}^2}$. With this choice, the
\textbf{I-AL} method from \cite{LanMon:15} performs  with our
analysis:
$$ \left \lceil \mathcal{O}\left(\frac{\norm{G}^{3/2} R_d^{1/2} D_U }{\epsilon^{3/2}}\right) + \mathcal{O}\left(\frac{L_f D_U}{\epsilon} \right) \right \rceil$$

\vspace{-10pt}

\noindent projections onto  $U$, for attaining an $\epsilon$-optimal
point w.r.t. optimality criteria \eqref{criteria_Lan}.

\vspace{3pt}

\noindent Moreover, using a straightforward modification of the
first stage of the \textbf{I-AL} method by replacing the outer dual
gradient method with an outer dual fast gradient method, we can
obtain a fast \textbf{I-AL} method. From Theorem
\ref{aug_fastgrad_outer}  we have that fast \textbf{I-AL} method
performs: $ N^{\text{out}}_{\epsilon,2} = \left\lceil \sqrt{\frac{8 R_d}{\mu
\epsilon}} \right\rceil$ outer iterations with inner accuracy $\delta =
\frac{\mu \epsilon^2}{128}$, to attain an $\epsilon-$optimal point
satisfying \eqref{criteria_Lan}. Using the same reasoning as in the
previous case, the first stage of fast \textbf{I-AL} method requires
$ N^{\text{out}}_{\epsilon,2} N_{\delta}^{\text{in}}$
projections onto $U$ and the \textbf{Postprocessing} procedure
performs $ \left\lceil \frac{2^{5/2} (L_f + \mu\norm{G}^2)D_U}{\epsilon} \right\rceil$
projections onto $U$. Using these bounds, for any $\mu \ge
\frac{L_f}{\norm{G}^2}$, the total number of projections required by
the fast  \textbf{I-AL} method is bounded by: $ \frac{2^{15/2}
R_d^{1/2} \norm{G}D_U}{\mu^{1/2} \epsilon^{3/2}} + \frac{2^{7/2} \mu
\norm{G}^2D_U}{\epsilon}$. In order to attain the optimal
complexity, we choose the smoothing parameter as $ \mu = \frac{4
R_d^{1/3}}{\epsilon^{1/3}\norm{G}^{2/3}} + \frac{L_f}{\norm{G}^2}$.
With this choice, the fast \textbf{I-AL} method performs with our
analysis:
$$ \left \lceil \mathcal{O}\left(\frac{\norm{G}^{4/3} R_d^{1/3} D_U}{\epsilon^{4/3}}\right) + \mathcal{O}\left(\frac{L_f D_U}{\epsilon}\right) \right \rceil$$

\vspace{-10pt}

\noindent projections onto $U$. In conclusion, based on our settings
we obtain computational complexities of order
$\mathcal{O}(\epsilon^{-\frac{3}{2}})$ for the original
\textbf{I-AL} method and of order
$\mathcal{O}(\epsilon^{-\frac{4}{3}})$ for the fast \textbf{I-AL}
method, which are significantly better than the estimate
$\mathcal{O}(\epsilon^{-\frac{7}{4}})$ given in \cite{LanMon:15} for
optimality criteria \eqref{criteria_Lan}. Moreover, in our
optimality criteria defined in Section 2, we have seen that for an
optimal smoothing parameter both classical and fast augmented
Lagrangian methods have the same complexity, while in optimality
criteria \eqref{criteria_Lan} the fast \textbf{I-AL} has the best
overall complexity.  Finally, we can combine our approach with a
regularization technique, i.e. the  addition of a strongly convex
term $\frac{\gamma}{2}\norm{u-u^0}^2$ to the objective function,
used e.g. in \cite{LanMon:15},  and obtain also computational
complexity (for the last primal point) of order
$\mathcal{O}(\epsilon^{-1})$ in optimality criteria
\eqref{criteria_Lan}. Due to space limitations we omit these
derivations.

\vspace{2pt}

\noindent Outer complexity estimate of order
$\mathcal{O}(\frac{1}{\epsilon})$ for fast gradient Nesterov type
smoothing methods were   derived e.g. in
\cite{BotHen:15,NecSuy:08,QuoNec:15}. From our previous analysis we
can conclude that for an adequate choice of the parameter $\mu$, the
number of outer iterations is only one, and therefore, the outer
complexity estimates are irrelevant to the total complexity of the
method. Thus, we need to derive overall complexities as we do in
this paper.

\vspace{3pt}

\noindent Finally, there are very few iteration complexity results
for first order methods  for convex problems that might not have a
Lagrange multiplier closing the duality gap.  Recently, Nesterov has
proposed a specialized subgradient method for solving directly
general nonsmooth convex problems with functional constraints
without assuming the existence of bounded  optimal Lagrange
multipliers \cite{Nes:14}. The specialized subgradient method in
\cite{Nes:14} requires $\mathcal{O}(\frac{1}{\epsilon^2})$ total
subgradient computations for either the objective function or for a
functional constraint. In \cite{LanMon:13} the classical quadratic
penalty scheme is combined with Nesterov optimal method for solving
a general conic problem, but under the strong assumption of the
existence of optimal Lagrange multipliers. If the objective function
is smooth, then the quadratic penalty method requires
$\mathcal{O}(\frac{1}{\epsilon^2})$ projections on the simple convex
set and on the cone to attain an $\epsilon$-solution satisfying a
criterion given in terms of a set of KKT conditions. On the other
hand, using a regularization strategy for the original problem, the
quadratic penalty method requires  $\mathcal{O}(\frac{1}{\epsilon}
\log \frac{1}{\epsilon})$ projections to attain $\epsilon$-solution
for the same criterion. Therefore, the assumption on the existence
of an optimal Lagrange multiplier improves the iteration complexity
of a quadratic penalty method  from
$\mathcal{O}(\frac{1}{\epsilon^{3/2}})$ (see Section 4) to
$\mathcal{O}(\frac{1}{\epsilon})$ total iterations.  Moreover, for
this particular setting, one can guarantee that the suboptimality
estimates hold in both sides with arbitrary accuracy, compared with
our setting where only the right hand side can be attained
arbitrarily small. In conclusion, the \textit{price} we pay for
tackling a more general conic convex problem is the additional
computational effort and the fact that the function value represents
a lower approximation of the optimal value.


%
%

\section{Appendix}
In this section we provide  proofs for Theorems \ref{aug_grad_outer}
and \ref{aug_fastgrad_outer}.

\vspace{10pt}

\noindent \textbf{Appendix A.1}
\begin{proof} [Proof of Theorem \ref{aug_grad_outer}]
We derive the sublinear estimates for  primal infeasibility  and
primal suboptimality for the average primal point $\hat{u}^k =
\frac{1}{k}\sum\limits_{j=1}^k u^j$, where $u^j= u_{\mu}(x^j)$ and
$x^j$ is generated by Algorithm
\textbf{ICFG}$(d^{\text{ag}}_{\mu},0,3\delta,2L_{\text{d}})$ with
$\theta_k = 1$ for all $k \ge 0$. First, given the definition of
$x^{j+1}$ in Algorithm \textbf{ICFG} we get:
\[   x^{j+1} = x^j + \frac{1}{2L_{\text{d}}} \nabla_x \mathcal{L}^{\text{ag}}_{\mu}(u^j,x^j)
\quad \forall j \geq 0.\]
Subtracting $x^j$ from both sides, adding up these inequalities for
$j=0:k-1$, we get:
\begin{align*}
\left\|\frac{1}{k}\sum_{j=0}^{k-1} \nabla_x
\mathcal{L}^{\text{ag}}_{\mu}(u^j,x^j)\right\| =
\frac{2L_{\text{d}}}{k}\norm{x^{k} - x^0}.
\end{align*}
Note that $\nabla_x \mathcal{L}^{\text{ag}}_{\mu}(u^j,x^j) = Gu^j +
g - \left[ Gu^j + g + \frac{1}{\mu}x^j \right]_{\mathcal{K}}$. Using
notation $z^j = \left[ Gu^j + g + \frac{1}{\mu}x^j
\right]_{\mathcal{K}}$, then $\frac{1}{k} \sum\limits_{j=0}^{k-1}
z^j \in \mathcal{K}$. This fact implies:
\begin{equation}
\label{feasibility_prelim}
 \text{dist}_{\mathcal{K}}(G\hat{u}^k+g) \le
 \left\| \frac{1}{k}\sum_{j=0}^{k-1} (Gu^j + g)
 - \frac{1}{k}\sum\limits_{j=0}^{k-1} z^j \right\|   = \frac{2L_{\text{d}}}{k}\norm{x^{k} - x^0}.
\end{equation}
It remains to bound $\norm{x^{k} -x^0}$. Using the iteration of
\textbf{ICFG}, for $x \in \mathcal{K}^*$, we get:
\begin{align}
  \norm{x^{k+1}-x}^2 
&  = \norm{x^k  -x}^2  + 2\langle x^{k+1} - x^k, x^{k+1}- x \rangle - \norm{x^{k+1} - x^k}^2 \nonumber\\
&  = \norm{x^k  -x}^2  + \frac{1}{L_{\text{d}}} \langle \nabla_x \mathcal{L}^{\text{ag}}_{\mu}(u^k,x^k) , x^k - x \rangle \nonumber\\
& \quad \quad + \frac{1}{L_{\text{d}}} \left(\langle \nabla_x
\mathcal{L}^{\text{ag}}_{\mu}(u^k,x^k), x^{k+1} - x^{k} \rangle -
L_{\text{d}}\norm{x^{k+1} - x^k}^2\right)
\label{rel_seq1}\\
& \le \norm{x^k  -x}^2 +
\frac{1}{L_{\text{d}}}(d^{\text{ag}}_{\mu}(x^{k+1}) -
d^{\text{ag}}_{\mu}(x)) + \frac{3\delta}{L_{\text{d}}} \quad \forall
k \ge 0.\nonumber
\end{align}
Taking $x=x^*$ in the last inequality and using an inductive
argument, then we get:
\begin{equation*}
 \norm{x^{k} - x^0} \le \norm{x^{k}-x^*} + \norm{x^0-x^*}
 \le 2\norm{x^0-x^*} + \sqrt{\frac{3k \delta}{L_{\text{d}}}}.
\end{equation*}
We substitute this bound into \eqref{feasibility_prelim} and we get
the estimate on primal  infeasibility:
\begin{equation}\label{feasibility_prelim1}
 \text{dist}_{\mathcal{K}}(G\hat{u}^k+g)
 \le  \frac{4L_{\text{d}}R_{\text{d}}}{k} +
 \frac{2L_{\text{d}}}{k}\sqrt{\frac{3k\delta}{L_{\text{d}}}}
 = \frac{4L_{\text{d}}R_{\text{d}}}{k} + \sqrt{\frac{12L_{\text{d}}\delta}{k}}.
\end{equation}
It remains to derive the estimates on primal suboptimality. First,
we observe that for any $u \in U$, we have $d_{\mu}(x) \le f^*$ and
the following identity holds:
\begin{align}\label{identity_lag}
\mathcal{L}^{\text{ag}}_{\mu}(u,x) - \langle \nabla_x
\mathcal{L}^{\text{ag}}_{\mu}(u,x), x\rangle  = f(u) +
\frac{\mu}{2}\norm{\nabla_x \mathcal{L}^{\text{ag}}_{\mu}(u,x)}^2.
\end{align}

\noindent Based  on the previous discussion, from \eqref{rel_seq1}
and \eqref{identity_lag} we derive that:
\begin{align*}
 & \norm{x^{k+1} - x}^2
\!\le\! \norm{x^k - x}^2 \!+\! \frac{1}{L_{\text{d}}}
 \left(d_{\mu}(x^{k+1}) \!-\! \mathcal{L}_{\mu}(u^k,x^k) \!+\!
 \langle \nabla_x \mathcal{L}_{\mu}(u^k,x^k), x^k - x \rangle
 \!+\! 3\delta \right)\\
 & \qquad \le \norm{x^k -x}^2 + \frac{1}{L_{\text{d}}} \left( f^* - f(u^k) -
 \frac{\mu}{2}\norm{\nabla_x \mathcal{L}^{\text{ag}}_{\mu}(u,x)}^2 + 3\delta - \langle
\nabla_x \mathcal{L}_{\mu}(u^k,x^k),x\rangle   \right).
 \end{align*}
Taking now $x=0$, and using an inductive argument over $j=0:k-1$, we
obtain:
\begin{equation}\label{right_subopt}
 f(\hat{u}^k) - f^* \le \frac{L_{\text{d}}\norm{x^0}^2}{k} + 3\delta.
\end{equation}
On the other hand, to bound below $f(\hat{u}^k) - f^*$ we proceed as
follows:
\begin{align}
f^* &= \min\limits_{u \in U, r \in \mathcal{K}} f(u) + \langle x^*, Gu+g -r \rangle \le f(\hat{u}^k) + \langle x^*, G\hat{u}^k+g - \left[G\hat{u}^k + g\right]_{\mathcal{K}} \rangle \nonumber\\
& \le f(\hat{u}^k) + \norm{x^*} \norm{G\hat{u}^k+g -
\left[G\hat{u}^k + g\right]_{\mathcal{K}}}  = f(\hat{u}^k) +
\norm{x^*} \text{dist}_{\mathcal{K}}
\left(G\hat{u}^k+g\right).\label{left_subopt}
 \end{align}
Combining \eqref{feasibility_prelim1} with \eqref{left_subopt} and
then with \eqref{right_subopt}, we obtain the estimate on primal
suboptimality stated in the theorem.
\end{proof}


\vspace{10pt}

\noindent \textbf{Appendix A.2}
\begin{proof} [Proof of Theorem \ref{aug_fastgrad_outer}]
We derive sublinear estimates for primal infeasibility and
suboptimality of the average primal point $\hat{u}^k =
\frac{1}{S^{\theta}_k} \sum\limits_{j=0}^{k-1} \theta_j u^j$, where
$u^j = u_\mu(x^j)$ and $x^j$ generated by Algorithm
\textbf{ICFG}$(d^{\text{ag}}_{\mu},0,3\delta,2L_{\text{d}})$ with
$\theta_{k+1} = \frac{1 + \sqrt{1 + 4 \theta_k^2 }}{2}$ for all $k
\geq 1$.  We observe that: $\frac{k+1}{2} \le \theta_k \le k$ and
$S_k^{\theta} =\theta_{k-1}^2$.  We denote $l^k = x^{k-1} +
\theta_{k}(x^k - x^{k-1})$ and recall that the following relation
has been proved in \cite{Tse:08,NecPat:14}:

\vspace{-25pt}

\begin{align}
\label{corr2_relation_ag} \theta_{k}^2 (d_{\mu}^{\text{ag}}(x) \!-\!
d_{\mu}^{\text{ag}}(x^{k})) \!+\!
\sum\limits_{i=1}^{k-1}\theta_{i}\Delta(x,y^i) \!+\! L_\text{d} &
\norm{l^{k} \!-\! x}^2  \!\le\! L_\text{d} \norm{x^0-x}^2 \!+\!
3\sum\limits_{i=1}^{k-1}\theta_i^2 \delta,
\end{align}

\vspace{-10pt}

\noindent where $\Delta(x,y) =
\mathcal{L}_{\mu}^{\text{ag}}(u_{\mu}(y),y) + \langle \nabla_x
\mathcal{L}_{\mu}^{\text{ag}}(u_{\mu}(y),y), x- y\rangle -
d_{\mu}^{\text{ag}}(x)$. Now we are ready to prove Theorem
\ref{aug_fastgrad_outer}. From definition of  augmented dual
function $d^{\text{ag}}_{\mu}$, it can be seen that  $x^k = y^k +
\frac{1}{2 L_{\text{d}}}\nabla_x
\mathcal{L}_{\mu}^{\text{ag}}(u^k,y^k)$. Multiplying by
$\theta_{k}$, we obtain:
\begin{align} \label{feasibility_aux1}
\frac{\theta_k}{2L_{\text{d}}} \nabla_x
\mathcal{L}_{\mu}^{\text{ag}}(u^k,y^k)
& =  \theta_{k}(x^{k}-y^{k}) = \theta_{k}(x^{k}-x^{k-1}) + (\theta_{k-1} -1)(x^{k-2} - x^{k-1})\nonumber\\
& = \underbrace{x^{k-1} + \theta_{k}(x^{k}-x^{k-1})}_{l^{k}} -
\underbrace{(x^{k-2} + \theta_{k-1}(x^{k-1}-x^{k-2}))}_{l^{k-1}}.
\end{align}
\noindent Summing on the history of $l^k$ and multiplying by
$\frac{2L_{\text{d}}}{S_k^\theta}$, we obtain:
\begin{align*}
\text{dist}_{\mathcal{K}}\left(G\hat{u}^k + g\right) \le
\left\|\sum\limits_{j=0}^{k-1}  \frac{\theta_j}{S_k^\theta} \nabla_x
\mathcal{L}_{\mu}^{\text{ag}}(u^j,y^j) \right\| =
\frac{2L_\text{d}}{S_k^\theta}\norm{l^{k}-l^0} \le
\frac{8L_\text{d}}{k^2} \norm{l^k -l^{0}}.
\end{align*}

\vspace{-10pt}

\noindent Since  $x^* = \arg\max\limits_{x \in \rset^m}
d^{\text{ag}}_{\mu}(x)$, by taking $x = x^*$ in
\eqref{corr2_relation_ag}, we get:
\begin{align*}
\| l^k - x^* \| \le \sqrt{\| x^0 - x^* \|^2 +
\sum\limits_{i=1}^{k-1} \frac{3\theta_i^2\delta}{L_{\text{d}}}}
&\le \| x^0 - x^* \| + \sqrt{\frac{3\delta}{L_{\text{d}}} S^{\theta}_{k} \max_{1\le i \le k-1} \theta_i } \\
& \le \| x^0 - x^* \| + \sqrt{\frac{3 \delta}{L_{\text{d}}}}
(k-1)^{3/2},
\end{align*}
for all $k \geq 0$. Thus, we can further bound the primal
feasibility as follows:
\begin{equation}
\label{infes_av_2_ag} \text{dist}_{\mathcal{K}}(G \hat{u}^k + g) \le
\frac{8L_{\text{d}} R_{\text{d}}}{k^2} +
8\sqrt{\frac{3L_{\text{d}}\delta}{k}}.
\end{equation}

\noindent Further, we derive sublinear  estimates for primal
suboptimality. First, note that:
\begin{align*}
\Delta(x,y^k) &=  \mathcal{L}_{\mu}^{\text{ag}}(u^{k},y^{k}) + \langle \nabla_x \mathcal{L}_{\mu}^{\text{ag}}(u^{k},y^{k}), x-y^{k}\rangle - d_{\mu}^{\text{ag}}(x) \\
& \ge f(u^k) + \langle \nabla_x \mathcal{L}^{\text{ag}}(u^k,y^k),
x\rangle - d_{\mu}^{\text{ag}}(x).
 \end{align*}
\noindent Summing on the history  and using  the convexity of $f$,
we get:
\begin{align}
\sum\limits_{i=1}^{k-1} \theta_i \Delta(x,y^i)
& \ge \sum\limits_{i=1}^{k-1}\theta_i( f(u^i) + \langle \nabla_x \mathcal{L}^{\text{ag}}(u^i,y^i), x\rangle - d_{\mu}^{\text{ag}}(x) )\nonumber\\
& \ge \theta_{k}^2\left(f(\hat{u}^k) + \sum\limits_{i=1}^{k-1}
\frac{\theta_i}{S_k^{\theta}} \langle \nabla_x
\mathcal{L}_{\mu}^{\text{ag}}(u^{i},y^i),x\rangle
-d_{\mu}^{\text{ag}}(x)\right), \label{sum_theta_aux_ag}
\end{align}
for all $x \in \rset^m$. Using \eqref{sum_theta_aux_ag} in
\eqref{corr2_relation_ag}, and dropping the term
$L_{\text{d}}\norm{l^{k}-x}^2$, we have:
\begin{equation*}
f(\hat{u}^k) + \sum\limits_{i=1}^{k-1} \frac{\theta_i}{S_k^{\theta}}
\langle \nabla_x \mathcal{L}_{\mu}^{\text{ag}}(u^{i},y^i),x\rangle
-d_{\mu}^{\text{ag}}(x) \overset{\eqref{sum_theta_aux_ag} +
\eqref{corr2_relation_ag}}{\le}
\frac{L_\text{d}}{\theta_{k-1}^2}\norm{x^0-x}^2 +
\frac{3\sum\limits_{i=1}^{k-1} \theta_i^2}{\theta_{k-1}^2} \delta
\end{equation*}
for all $x \in \rset^m.$ Given that
$\frac{1}{\theta_{k-1}^2}\sum\limits_{i=1}^{k-1} \theta_i^2 =
\frac{1}{S^{\theta}_{k}} \sum\limits_{i=1}^{k-1} \theta_i^2 \le
\max\limits_{1\le i \le k-1} \theta_i \le k-1$ and
$d_{\mu}^{\text{ag}}(x) \le f^*$, by choosing the Lagrange
multiplier $x = 0$, we further have:
\begin{align}\label{subopt_av_1_ag}
f(\hat{u}^{k}) - f^* \le f(\hat{u}^{k}) - d_{\mu}^{\text{ag}}(0) \le
\frac{4L_\text{d}\norm{x^0}^2}{k^2} + 3k\delta.
\end{align}
\noindent On the other hand, we have:
\begin{align}
f^* &=  \min_{u \in U, s \in \mathcal{K}} f(u) + \langle x^*, Gu+g -s \rangle \leq f(\hat u^k) + \langle x^*, G\hat u^k +g - \left[Gu^k+g \right]_{\mathcal{K}} \rangle\nonumber\\
& \overset{\eqref{infes_av_2_ag}}{\leq} f(\hat x^k) +
\frac{8L_{\text{d}} R_{\text{d}}^2}{k^2} +
8R_{\text{d}}\sqrt{\frac{3L_{\text{d}}\delta}{k}}.
\label{subopt_av_2_ag}
\end{align}
Finally, from \eqref{infes_av_2_ag}, \eqref{subopt_av_1_ag} and
\eqref{subopt_av_2_ag} we get the estimates on primal infeasibility
and suboptimality stated in the theorem.
\end{proof}


\vspace{10pt}

\noindent\textbf{Appendix A.3}
\begin{proof}[of Theorem \ref{th_modif_ns}]
\noindent  First, note that an analog result as in the previous Appendix 
holds in this case, and for clarity we state
it below (see e.g. \cite{Tse:08} for a proof):
\begin{lemma}
\label{th_tseng} Let $\mu,\delta>0$ and sequences $(x^k, y^k)_{k
\geq 0}$ be generated by Algorithm
\textbf{ICFG}($d_{U,\mu},d_{\mathcal{K}},\delta$) with $\theta_{k+1}
= \frac{1 + \sqrt{1 + 4 \theta_k^2 }}{2}$ for all $k \geq 1$, then
for any Lagrange multiplier $x$ and iteration $k$ we have:
\begin{equation}\label{corr2_relation}
\theta_{k}^2 (d_{\mu}(x) - d_{\mu}(x^{k})) +
\sum\limits_{i=1}^{k-1}\theta_{i}\Delta(x,y^i) +
L_\text{d}\norm{l^{k}-x}^2 \le
L_\text{d} \norm{x^0-x}^2 + 3\sum\limits_{i=1}^{k}\theta_i^2 \delta,
\end{equation}
where we use $\Delta(x,y) = \mathcal{L}_{\mu}(u_{\mu}(y),y) +
\langle \nabla_x \mathcal{L}_{\mu}(u_{\mu}(y),y), x- y\rangle -
d_{\mu}(x)$.
\end{lemma}

\noindent Based on the same notations and reasoning as in Appendix
A.2,  taking $x=x^*_{\mu}$ in \eqref{corr2_relation} and using
 that  terms $\theta_{k}(f^*_{\mu}-d_{\mu}(x^{k}))$ and $\sum\limits_{i=1}^{k-1}\theta_{i}
 \Delta(x^*_{\mu},y^i)$ are positive, we obtain:
\begin{align}
\text{dist}_{\mathcal{K}}\left(G\hat{u}^k +g \right)   &\le \frac{8L_\text{d}}{k^2} \|l^k - l^0\| \le \frac{8L_\text{d}}{k^2}(\|l^k - x^*_{\mu} \| + \|l^0 - x^*_{\mu}\|)  \nonumber\\
& \le \frac{8L_\text{d}R_{\text{d}}}{k^2}  +
8\sqrt{\frac{3L_{\text{d}}\delta}{k}}.
\label{infes_av}
\end{align}

\noindent Further, we derive sublinear  estimates for primal
suboptimality. First, note that:
\begin{align*}
\Delta(x,y^{k}) &=  \mathcal{L}_{\mu}(u^{k},y^{k}) + \langle \nabla_x \mathcal{L}_{\mu}(u^{k},y^{k}), x-y^{k}\rangle - d_{\mu}(x) \\
&=\mathcal{L}_{\mu}(u^{k},y^{k}) + \langle Gu^{k}+g , x-y^{k}\rangle -
 d_{\mu}(x)
=\mathcal{L}_{\mu}(u^{k},x) - d_{\mu}(x).
 \end{align*}
\noindent Summing on the history  and using  the convexity of
$\mathcal{L}_{\mu}(\cdot,x)$, we get:
\begin{align}
\sum\limits_{i=1}^{k-1}\theta_i \Delta(x,y^i)
&= \sum\limits_{i=1}^{k-1}\theta_i(\mathcal{L}_{\mu}(u^{i},x) - d_{\mu}(x)) \nonumber\\
&\ge S_{k}^{\theta}\left(\mathcal{L}_{\mu}(\hat{u}^{k},x) -
d_{\mu}(x)\right) =
\theta_{k}^2\left(\mathcal{L}_{\mu}(\hat{u}^{k},x) -
d_{\mu}(x)\right). \label{sum_theta_aux}
\end{align}
Using \eqref{sum_theta_aux} in \eqref{corr2_relation},
$\frac{\sum\limits_{i=1}^{k-1}\theta_{i}^2}{S_{k}^{\theta}} \le
\max\limits_{1 \le i \le k-1} \theta_i \le k-1$ and dropping  term
$\frac{L_{\text{d}}}{2}\norm{l^{k}-x}^2$, we have:
\begin{equation}\label{obs_aux}
\mathcal{L}_{\mu}(\hat{u}^{k}, x) - d_{\mu}(x^{k}) \le
\frac{L_\text{d}}{\theta_{k}^2}\norm{x^0-x}^2 + 3k\delta.
\end{equation}
Choosing the  multiplier $x =0$, we observe that
$\mathcal{L}_{\mu}(\hat{u}^k,0) \ge f(\hat{u}^k)$ and $ d_{\mu}(x)
\le f^* + \frac{\mu}{2}D_U^2$ for all $x \in - \mathcal{K}^*$. Then,
combining this observations with \eqref{obs_aux} leads to:
\begin{align*}
f(\hat{u}^{k}) - f^* \le f(\hat{u}^{k}) - d_{\mu}(x^{k}) &
\overset{\eqref{obs_aux}}{\le}
\frac{4L_\text{d}\norm{x^0}^2}{k^2} + \frac{\mu}{2}D_U^2 + 3k
\delta
& \le \frac{4\norm{G}^2 R_d^2}{\mu k^2} + \frac{\mu}{2}D_U^2 +
3k\delta.
\end{align*}
We choose the optimal smoothing parameter by minimizing the above
expression over $\mu$ and  obtain: $\mu (k) = \frac{2^{3/2}
\norm{G}R_{\text{d}}}{D_Uk}$.  Replacing this value in the above
estimates, we obtain:
\begin{align*} 
f(\hat{u}^{k}) - f^* &\le  \frac{2^{3/2}\norm{G} R_{\text{d}} D_U}{k} +
3k\delta.
\end{align*}
Also, taking $L_{\text{d}} = \frac{\norm{G}^2}{\mu(k)}$ in the
feasibility gap \eqref{infes_av}, we get the estimate on
infeasibility: $\text{dist}_{\mathcal{K}}\left( G\hat{u}^k +g
\right) \le \frac{2^{3/2}\norm{G} D_U}{k}   +  2 \left(\frac{ \norm{G} D_U
\delta }{R_{\text{d}}}\right)^{1/2}$.
On the other hand, we have:
\begin{align*}
f^* &=  \min_{u \in U, s \in \mathcal{K}} f(u) + \langle x^*, Gu+g
-s \rangle \leq f(\hat u^k) + \langle x^*,
G\hat u^k +g - \left[Gu^{k}+g \right]_{\mathcal{K}} \rangle\nonumber\\
& \leq f(\hat u^{k}) + \frac{2 \norm{G}D_U R_d}{k} +2\left( \delta
\norm{G} D_U R_d
\right)^{1/2}, 
\end{align*}
which proves the statements of the theorem.
\end{proof}


\begin{thebibliography}{35}

\bibitem{AybIye:13}
N.  Aybat and G. Iyengar, \textit{An Augmented Lagrangian Method for
Conic Convex Programming}, Working paper,
{ \verb http://arxiv.org/abs/1302.6322 }, 2013.

\bibitem{BecTeb:09}
A.~Beck and M.~Teboulle, \textit{A Fast Iterative
Shrinkage-Thresholding Algorithm for Linear Inverse Problems}, SIAM
Journal Imaging Science, 2(1): 183--202, 2009.

\bibitem{BotHen:15}
R.  Bot and C. Hendrich, \textit{ A variable smoothing algorithm for
solving convex optimization problems}, TOP, 23(1): 124--150, 2015.

\bibitem{BotHen:15o}
R. Bot and C. Hendrich, \textit{On the acceleration of the double
smoothing technique for unconstrained convex optimization problems},
Optimization, 64(2): 265--288, 2015.

\bibitem{BecTeb:12}
A. Beck and M. Teboulle, \textit{Smoothing and first order methods:
a unified framework}, SIAM Journal on Optimization, 22(2): 557--580,
2012.

\bibitem{DevGli:14}
O. Devolder, F. Glineur and Yu. Nesterov, \textit{First-order
methods of smooth convex optimization with inexact oracle},
Mathematical Programming, 146: 37--75, 2014.


\bibitem{DevGli:12}
O. Devolder, F. Glineur and Yu. Nesterov, \textit{Double smoothing
technique for large-scale linearly constrained convex optimization},
SIAM Journal on Optimization, 22(2): 702–-727, 2012.

\bibitem{LanLu:11}
G. Lan, Z. Lu   and  R.  Monteiro, \textit{Primal-dual first order
methods with $\mathcal{O}(1/\epsilon)$ iteration-complexity for cone
programming}, Mathematical Programming, 126: 1--29, 2011.

\bibitem{LanMon:13}
G. Lan  and R. Monteiro, \textit{Iteration-complexity of first order
penalty methods for convex programming}, Mathematical Programming,
138: 115--139, 2013.

\bibitem{LanMon:15}
G. Lan  and R.  Monteiro, \textit{Iteration-complexity of first
order augmented Lagrangian methods for convex programming},
Mathematical Programming, DOI 10.1007/s10107-015-0861-x, 2015.


\bibitem{MonSva:10}
R. Monteiro and B. Svaiter, \textit{On the complexity of the hybrid
proximal extragradient method for the iterates and the ergodic
mean}, SIAM J. Optimization, 20(6): 2755--2787, 2010.


\bibitem{NecNed:13}
I.~Necoara and V.~Nedelcu, \textit{Rate analysis of inexact dual
first order methods: application to  dual decomposition}, IEEE
Transactions on Automatic Control, 59(5): 1232--1243, 2014.

\bibitem{NecPat:14}
I. Necoara and A. Patrascu, \textit{Iteration complexity analysis of
dual first order methods for conic convex programming}, Optimization
Methods and Software, 31(3): 645--678, 2016.

\bibitem{NecPat:15}
I. Necoara, A. Patrascu, F. Glineur, \textit{Complexity
certifications of first order inexact Lagrangian and penalty methods
for conic convex programming}, Technical Report, University
Politehnica of Bucharest, 2015,
 {\verb https://arxiv.org/abs/1506.05320 }

\bibitem{NecSuy:08}
I.~Necoara and J.A.K. Suykens, \textit{Application of a smoothing
technique to decomposition in convex  optimization}, IEEE
Transactions on Automatic Control, 53(11): 2674--2679, 2008.

\bibitem{NedNec:14}
V. Nedelcu, I. Necoara and Q. Tran-Dinh, \textit{Computational
Complexity of Inexact Gradient Augmented Lagrangian Methods:
Application to Constrained MPC}, SIAM Journal on Control and
Optimization, 52(5): 3109--3134, 2014.

\bibitem{NedOzd:09}
A. Nedic and A. Ozdaglar, \textit{Approximate Primal Solutions and
Rate Analysis for Dual Subgradient Methods}, SIAM Journal on
Optimization 19(4): 1757--1780, 2009.


\bibitem{Nem:04}
A. Nemirovski, \textit{Prox-method with rate of convergence
$\mathcal{O}(1/t)$ for variational inequalities with Lipschitz
continuous monotone operators and smooth convex-concave saddle point
problems}, SIAM Journal on Optimization, 15(1): 229--251, 2004.



\bibitem{Nes:051}
Yu. Nesterov, \textit{Smooth minimization of non-smooth functions},
Mathematical Programming, 103: 127--152, 2005.

\bibitem{Nes:07}
Yu. Nesterov, \textit{Dual extrapolation and its applications to
solving variational inequalities and related problems}, Mathematical
Programming, 109: 319--344.

\bibitem{Nes:13}
Yu. Nesterov, \textit{Gradient methods for minimizing composite
functions}, Mathematical Programming, 140: 125--161, 2013.

\bibitem{Nes:14}
Yu. Nesterov, {\em Subgradient methods for huge-scale optimization
problems}, Mathematical Programming, 146: 275--297, 2014,
 {\verb www.imtlucca.it/embopt-14/Slides/nesterov.pdf }.

\bibitem{QuoNec:15}
Q. Tran-Dinh, I. Necoara and M. Diehl, \textit{Fast inexact
distributed optimization algorithms for separable convex
optimization}, Optimization, 65(2): 325–356, 2016.

\bibitem{QuoCev:14}
Q. Tran-Dinh, V. Cevher, \textit{A primal-dual algorithmic framework
for constrained convex minimization}, Technical report, 2014,
 { \verb http://arxiv.org/abs/1406.5403 }.


\bibitem{RocWet:98}
R.T. Rockafellar and R. Wets. \textit{Variational Analysis},
Springer, 1998.

\bibitem{Tse:08}
P.~Tseng, \textit{On accelerated proximal gradient methods for
convex-concave  optimization},  SIAM Journal on Optimization
(submitted), 2008.

\bibitem{YurQuo:15}
A. Yurtsever, Q. Tran-Dinh  and V. Cevher, \textit{Universal
Primal-Dual Proximal-Gradient Methods}, Technical report, 2015,
 { \verb http://arxiv.org/abs/1502.03123 }.
\end{thebibliography}
\end{document}